\documentclass[12pt,a4paper]{amsart}
\usepackage{amsmath}
\usepackage{amsthm}
\usepackage{amssymb}
\usepackage{amscd}
\usepackage{drpack}
\usepackage[english]{babel}
\usepackage{verbatim}
\usepackage[shortlabels]{enumitem}
\usepackage{tikz-cd}
\usepackage{mathrsfs}
\usepackage{hyperref}

\newtheoremstyle{mio}%
{}{} 
{\itshape}{} 
{\bfseries}{.}{ } 
{#1 #2\thmnote{~\mdseries(#3)}} 
\theoremstyle{mio}
\newtheorem{teor}{Theorem}[section]
\newtheorem{cor}[teor]{Corollary}
\newtheorem{prop}[teor]{Proposition}
\newtheorem{lemma}[teor]{Lemma}
\newtheorem{defin}[teor]{Definition}

\newtheoremstyle{definition2}%
{}{} 
{}{} 
{\bfseries}{.}{ } 
{#1 #2\thmnote{\mdseries~ #3}} 
\theoremstyle{definition2}
\newtheorem{ex}[teor]{Example}
\newtheorem{oss}[teor]{Remark}

\newcommand{\funct}{\mathcal{F}}
\newcommand{\dirsum}{\funct_s}
\newcommand{\functbound}{\funct_b}
\newcommand{\supp}{\mathrm{supp}}
\newcommand{\residuo}{\kappa}

\newcommand{\Inv}{\mathrm{Inv}}
\newcommand{\V}{\mathcal{V}}
\newcommand{\D}{\mathcal{D}}

\newcommand{\deriv}{\mathcal{D}}
\newcommand{\derivX}{\mathcal{D}_0}
\newcommand{\cb}{\mathrm{CB}}
\newcommand{\mmax}{\mathcal{M}}
\newcommand{\Kr}{\mathrm{Kr}}
\newcommand{\Specast}{\Spec^\ast}

\title[Groups of invertible ideals]{Groups of invertible ideals of one-dimensional Pr\"ufer domains as groups of integer-valued functions}
\author{Dario Spirito}
\address{Dipartimento di Scienze Matematiche, Fisiche e Informatiche, Universit\`a di Udine, Udine, Italy}
\email{dario.spirito@uniud.it}

\begin{document}
\begin{abstract}
Let $G$ be a one-dimensional $\ell$-subgroup of the group $\funct(X,\insZ)$ of integer-valued functions on a set $X$. We show that $G$ is free under some hypothesis on the spectrum of $G$ and on its quotient groups at the prime ideals. We translate this result in the context of the study of freeness of the group $\Inv(D)$ of invertible ideals of a Pr\"ufer domain $D$: in particular, we introduce the class of \emph{dd-domains} as the class of Pr\"ufer domains having a set $X$ that is dense in $\Spec(D)$ (with respect to the inverse topology) and whose localizations are DVRs. This class is exactly the class of Pr\"ufer domains for which $\Inv(D)$ is isomorphic (as an $\ell$-group) to a subgroup of $\funct(X,\insZ)$.
\end{abstract}
\maketitle

\section{Introduction}
Let $X$ be a nonempty set, and let $\funct(X,\insZ)$ be the set of all functions $X\longrightarrow\insZ$. Then, a classical theorem by Baer \cite{baer-dirprodZ} states that $\funct(X,\insZ)$ is not a free group; on the other hand, the subgroup $\functbound(X,\insZ)$ of all bounded functions is free \cite{specker}. In \cite{SP-scattered} and \cite{bounded-almded}, the latter result was used to show that the group $\Inv(D)$ of invertible ideals of an almost Dedekind domain $D$ (see Section \ref{sect:InvD} below for the definition) is always free. A similar approach was used to study $\Inv(D)$ in the more general case of strongly discrete Pr\"ufer domains (see \cite{InvXD}), obtaining sufficient conditions for $\Inv(D)$ to be free.

An interesting feature of the almost Dedekind domain case is that there is a natural injective map $\Phi:\Inv(D)\longrightarrow\funct(\Max(D),\insZ)$, sending an ideal $I$ to its \emph{ideal function} $\nu_I$; in particular, if every ideal function is bounded, then $\Inv(D)$ is free. However, the image of $\Inv(D)$ is in general not contained in $\functbound(X,\insZ)$: therefore, in order to prove the freeness of $\Inv(D)$, it was necessary to set up an ascending chain of domains (connected to a descending chain of subsets of the maximal space reminiscent of the construction of the derived set of a topological space) that allowed to express $\Inv(D)$ as a smooth union of free groups such that the successive quotients are representable (non-canonically) as subgroups of $\functbound(\Max(D),\insZ)$. 

Even when not every ideal function is bounded, the image $\Phi(\Inv(D))$ of $\Inv(D)$ is not an arbitrary subgroup of $\funct(\Max(D),\insZ)$, but satisfies a ``bounded ratio'' property: if $f,g\in\Phi(\Inv(D))^+$ have the same cozero set, there is a positive integer $n$ such that $nf\geq g$ (this corresponds to the fact that the ideals corresponding to $f$ and $g$ are finitely generated ideals with the same radical). Moreover, $\Phi(\Inv(D))$ is an $\ell$-group whose spectrum is one-dimensional.

In this paper, we study $\ell$-subgroups of $\funct(X,\insZ)$ (where $X$ is an arbitrary infinite set), as a generalization of the $\ell$-groups arising from almost Dedekind domain, with particular emphasis on the one-dimensional case. In the ring case, such subgroups correspond to a class of Pr\"ufer domains that we call \emph{dd-domains}: they can be defined in a pure ring-theoretical way as those Pr\"ufer domains whose set of discrete primes is dense in the spectrum with respect to the inverse topology ($P$ is a discrete prime if $D_P$ is a DVR).

The main group-theoretical result of this paper is Theorem \ref{teor:basic}, which deals with one-dimensional $\ell$-groups $G\subseteq\funct(X,\insZ)$ with a single ``infinite'' prime (see Section \ref{setc:Specinfty} for a precise definition); we show that such a group is free under some mild topological hypothesis on the spectrum of $G$ and an algebraic hypothesis on the quotient group by the infinite prime. As infinite primes of $\ell$-groups correspond (roughly) to non-discrete primes of Pr\"ufer domains, the rings corresponding to this case are one-dimensional Pr\"ufer domains $D$ with a single non-discrete prime ideal, which is also not isolated (Theorem \ref{teor:invfree-1ptolim}).

We then generalize this result to the case where the set of infinite primes (or, in the domain case, the set of non-discrete prime ideals) is countable and none of its points is in the perfect hull of the spectrum (see Theorems \ref{teor:invfree-countptilim} and \ref{teor:grp-countptilim}). These results allow to prove that $\Inv(D)$ is free in many cases even outside the strongly discrete case considered in \cite{InvXD} (of which almost Dedekind domains are the one-dimensional case).

\section{Preliminaries and notation}
\subsection{Pr\"ufer domains}
We refer to \cite{gilmer} for facts about valuation and Pr\"ufer domains.

Let $V$ be an integral domain with quotient field $K$. Then, $V$ is a \emph{valuation domain} if, for every $x\in K$, either $x\in V$ or $x^{-1}\in V$. A valuation domain is always local, and all its finitely generated ideals are principal; all its localizations are valuation domains too.

To every valuation domain $V$ we can associate a function $v$ from $K\setminus\{0\}$ to a totally ordered abelian group $\Gamma$ (called a \emph{valuation}) such that, for every $x,y\in K\setminus \{0\}$,
\begin{itemize}
\item $v(xy)=v(x)+v(y)$;
\item $v(x+y)\geq\min\{v(x),v(y)\}$ (if $x+y\neq 0$);
\item $x\in V$ if and only if $v(x)\geq 0$.
\end{itemize}
If $v$ is surjective, then $\Gamma$ is called the \emph{value group} of $V$. Moreover, if the Krull dimension of $V$ is $1$, then the value group is always isomorphic to a subgroup of the additive group $\insR$; thus, in this case, we can consider $v$ as a function from $K\setminus\{0\}$ to $\insR$.

If $V$ is Noetherian, then $V$ is called a \emph{discrete valuation ring} (DVR) and its value group is isomorphic to $\insZ$. In this case, when $v$ is seen as a map to $\insR$, we shall normalize $v$ by setting $v(x)=1$ if $x$ generates the maximal ideal of $V$.

A domain $D$ is a \emph{Pr\"ufer domain} if it is locally a valuation domain, i.e., if $D_M$ is a valuation domain for every maximal ideal $M$. A \emph{B\'ezout domain} is an integral domain where every finitely generated ideal is principal; every B\'ezout domain is Pr\"ufer.

An \emph{almost Dedekind domain} is an integral domain such that $D_M$ is a DVR for every maximal ideal $M$. A \emph{Dedekind domain} is a Noetherian Pr\"ufer domain; equivalently, it is an almost Dedekind domain that is locally finite (i.e., each $x\in D\setminus\{0\}$ is contained in only finitely many maximal ideals).

If $D$ is a Pr\"ufer domain and $P$ is a prime ideal such that $D_P$ is a DVR, we say that $P$ is a \emph{discrete prime} of $D$.

\subsection{Invertible ideals}\label{sect:InvD}
A \emph{fractional ideal} of a domain $D$ with quotient field $K$ is a $D$-submodule $I$ of $K$ such that $dI\subseteq D$ for some $d\in D$, $d\neq 0$. A fractional ideal $I$ is \emph{invertible} if there is a fractional ideal $J$ such that $IJ=D$. The set $\Inv(D)$ of the invertible ideals of $D$ is a group under multiplication, with identity $D$.

Every invertible ideal is finitely generated. If $D$ is a Pr\"ufer domain, then the converse holds, i.e., every finitely generated ideal is invertible.

When $D$ is a B\'ezout domain, $\Inv(D)$ is naturally isomorphic to the \emph{group of divisibility} of $D$, i.e., to the quotient between the multiplicative group of the quotient field of $D$ and the group of units of $D$. When $V$ is a valuation domain, $\Inv(V)$ is naturally isomorphic to the value group of $V$.

\subsection{Free groups}
Let $G$ be an abelian group. Then, $G$ is \emph{free} if it has a basis, i.e., if there is a set $E\subset G$ such that every $g\in G$ can be written uniquely as a sum of elements in $E$. Equivalently, $G$ is free if it is isomorphic to the direct sum of a family of copies of $\insZ$.

A subgroup of a free group is always free. Moreover, free groups are projective, i.e., if 
\begin{equation*}
0\longrightarrow A\longrightarrow B\longrightarrow C\longrightarrow 0
\end{equation*}
is an exact sequence and $C$ is free, then the sequence splits.

If $G$ is a group, a \emph{smooth sequence} of subgroups of $G$ is a well-ordered ascending chain $\{H_\alpha\}_{\alpha<\lambda}$ of subgroups of $G$ such that $H_\beta=\bigcup_{\alpha<\beta}H_\alpha$ for all limit ordinals $\beta$.

The main method we shall use to prove freeness of groups is the following result:
\begin{teor}\label{teor:free-chain}
\cite[Chapter 3, Lemma 7.4]{fuchs-abeliangroups} Let $G$ be a group and $\{H_\alpha\}_{\alpha<\lambda}$ be a smooth sequence of subgroups such that $H_0=0$ and $\bigcup_{\alpha<\lambda}H_\alpha=G$. If $H_{\alpha+1}/H_\alpha$ is free for every $\alpha$, then $G$ is free.
\end{teor}

\begin{oss}
The statement of \cite[Chapter 3, Lemma 7.4]{fuchs-abeliangroups} includes the hypothesis that each $H_\alpha$ is pure in $G$, that is, that $nG\cap H_\alpha=nH_\alpha$ for every $n\inN$. This hypothesis is not necessary, as the proof does not make use of it.

\end{oss}

\subsection{$\ell$-groups}
In the following, and throughout the paper, we shall only consider abelian groups. We refer to \cite{darnel-lgroups} for terminology about $\ell$-groups.

An \emph{$\ell$-group} (short for \emph{lattice-ordered group}) is a partially ordered group that is also a lattice, i.e., where every pair of elements has an infimum and a supremum. If $G$ is an $\ell$-group with identity $0$, then $G^+:=\{h\in G\mid h>0\}$ is the \emph{positive cone} of $G$.

An \emph{ideal} of $G$ is a convex $\ell$-subgroup of $G$, that is, a subgroup $I$ of $G$ such that, if $a,b\in I$, then $a\wedge b\in I$ and if $a\leq c\leq b$ then $c\in I$. An ideal $I$ is \emph{prime} if, for every $a,b\in G^+$ such that $a\wedge b\in I$, we have $a\in I$ or $b\in I$.



If $I$ is a prime ideal of $G$, then $G/I$ is a totally ordered group; we call it the \emph{residue group} at $I$. 

The \emph{spectrum} $\Spec(G)$ of $G$ is the set of all prime ideals of $G$. The \emph{dimension} of $G$ is the supremum of the lengths of the chains in $\Spec(G)$; in particular, $G$ has dimension $1$ (i.e., it is \emph{one-dimensional}) if no two prime filters properly contained in $G$ are comparable. We denote by $\Specast(G)$ the set $\Spec(G)\setminus\{G\}$.

\subsection{Pr\"ufer domains and $\ell$-groups}\label{sect:prufer-ell}
Let $D$ be a Pr\"ufer domain. Then, the group $\Inv(D)$ of the invertible fractional ideals of $D$ is an $\ell$-group under reverse containment \cite[Theorem 2(a)]{brewer-klinger}; conversely, every lattice-ordered group is isomorphic to $\Inv(D)$ for some B\'ezout domain $D$ (see \cite[Chapter III, Theorem 5.3]{fuchs-salce} or \cite[Theorem 11.2]{anderson-feil}; note that, for B\'ezout domains, $\Inv(D)$ is isomorphic to the group of divisibility of $D$). 

Let $D$ be a B\'ezout domain. By \cite[Theorem 11.3 and subsequent discussion]{anderson-feil}, the prime ideals of $\Inv(D)$ are the sets in the form $H_P:=\{I\in\Inv(D)\mid ID_P=D_P\}$, for some nonzero prime ideal $P$ of $D$; therefore, there is a bijection
\begin{equation*}
\begin{aligned}
\Psi\colon\Spec(D)  & \longrightarrow \Spec(\Inv(D)),\\
P & \longrightarrow H_P
\end{aligned}
\end{equation*}
that reverses the ordering. In particular, $\dim(D)=\dim(\Inv(D))$.

The convex group $H_P$ can also be seen as the kernel of the extension map $\Inv(D)\longrightarrow\Inv(D_P)$, $I\mapsto ID_P$. Since this map is surjective, $\Inv(D_P)\simeq\Inv(D)/H_P$, i.e., the value group of $D_P$ is isomorphic to the residue group of $\Inv(D)$ at $H_P$.

If $D$ is Pr\"ufer, but not necessarily B\'ezout, we can always reduce to the B\'ezout case by using the Kronecker function ring construction (see \cite[Chapter 32]{gilmer} for the definition and properties): by \cite[Lemma 3.5]{HK-Olb-Re}, if $D$ is a Pr\"ufer domain the extension map $\Psi:\Inv(D)\longrightarrow\Inv(\Kr(D))$, $I\mapsto I\Kr(D)$ is an isomorphism of $\ell$-groups.

Therefore, we can always transfer problems from $\ell$-group to B\'ezout domains, and considering $\Inv(D)$ in the case of B\'ezout domains is enough to study it for arbitrary Pr\"ufer domains.

\subsection{The inverse and constructible topologies}\label{sect:topologies}
Let $A$ be a ring, and let $\Spec(A)$ be the spectrum of $A$. For every $f\in A$, let
\begin{equation*}
\V(f):=\{P\in\Spec(A)\mid f\in P\}.
\end{equation*}
Unless otherwise specified, we shall always endow $\Spec(A)$ with the \emph{inverse topology}, i.e., the topology generated by the family $\{V(f)\mid f\in A\}$; its name comes from the fact that the order it induces on $\Spec(A)$ is the opposite order of the containment (which is the order induced by the usual Zariski topology). The spectrum of $A$ is compact under the inverse topology. The subspace topology on the set $\Max(A)$ of the maximal ideals of $A$ is Hausdorff, and the sets $\V(f)\cap\Max(A)$ are clopen and compact in $\Max(A)$. See \cite[Chapter 1]{spectralspaces-libro} for a deeper study of the inverse topologies and its connection with the Zariski and the constructible topology.

When $G$ is an $\ell$-group, we endow $\Spec(G)$ with the hull-kernel topology: it is generated by the sets $\D(f):=\{P\in\Spec(G)\mid f\notin P\}$ with $f\in G^+$ (or, equivalently, with $f\in G$). If $D$ is a Pr\"ufer domain and $G=\Inv(D)$, then the bijection
\begin{equation*}
\Psi:\Spec(D)\longrightarrow\Spec(\Inv(D))
\end{equation*}
becomes a homeomorphism when the $\Spec(D)$ is endowed with the inverse topology and $\Spec(\Inv(D))$ is endowed with the hull-kernel topology: indeed, if $f\in D$ then $\Psi(\V(f))=\D(fD)$ and if $I\in\Inv(D)$ then $I=(f_1,\ldots,f_n)$ is finitely generated and $\Psi^{-1}(\D(I))=\V(f_1)\cap\cdots\cap\V(f_n)$.

If $D$ is one-dimensional, $\Max(D)=\Spec(D)\setminus\{(0)\}$, and its image under the map $\Psi$ is $\Specast(\Inv(D))=\Spec(\Inv(D))\setminus\{\Inv(D)\}$.

\subsection{Groups of functions}
Let $X$ be a set. We denote by $\funct(X,\insZ)$ the group of all functions $X\longrightarrow\insZ$; this group is an $\ell$-group under the natural pointwise ordering (i.e., $f\leq g$ if and only if $f(x)\leq g(x)$ for all $x\in X$). In particular, $(f\wedge g)(x)=\inf\{f(x),g(x)\}$ and $(f\vee g)(x)=\sup\{f(x),g(x)\}$ for all $f,g\in\funct(X,\insZ)$.

If $f\in\funct(X,\insZ)$, the \emph{support} of $f$ is
\begin{equation*}
\supp(f):=\{x\in X\mid f(x)\neq 0\}.
\end{equation*}

For every $x\in X$, we denote by $e_x$ the \emph{basis element} at $x$, i.e., the function such that
\begin{equation*}
e_x(y)=\begin{cases}
1 & \text{if~}x=y\\
0 & \text{if~}x\neq y
\end{cases}
\end{equation*}
We set $\dirsum(X,\insZ):=\langle e_x\mid x\in X\rangle$, and we usually refer to it as the \emph{direct sum} (inside the direct product $\funct(X,\insZ)=\prod_{x\in X}e_x\insZ$).

\section{Primes at infinity}\label{setc:Specinfty}
We are interested in studying the spectrum of an $\ell$-group $G\subseteq\funct(X,\insZ)$. To do so, we divide the prime ideals of $G$ into two classes.
\begin{prop}\label{prop:finite}
Let $G\subseteq\funct(X,\insZ)$ be an $\ell$-group and let $x\in X$.
\begin{enumerate}[(a)]
\item\label{prop:finite:filter} The set $P_x:=\{f\in G\mid f(x)=0\}$ is a prime ideal of $G$.
\item\label{prop:finite:residue} If $P_x\neq G$, the residue group $G/P_x$ is isomorphic to $\insZ$.
\item\label{prop:finite:dense} The set $\mathcal{P}:=\{P_x\mid x\in X, P_x\neq G\}$ is dense in $\Spec(G)$.
\end{enumerate}
\end{prop}
\begin{proof}
It is straightforward to see that $P_x$ is convex; if $f\wedge g\in P_x$, then $0=(f\wedge g)(x)=\inf\{f(x),g(x)\}$ and thus one of $f(x)$ and $g(x)$ is $0$, i.e., $f\in P_x$ or $g\in P_x$. Hence $P_x$ is prime.

The ideal $P_x$ is the kernel of the evaluation map $v_x:G\longrightarrow\insZ$, $v_x(f):=f(x)$; if $P_x\neq G$, then $v_x(G)$ is not just $(0)$, and thus $G/P_x$ is isomorphic to a nonzero subgroup of $\insZ$, and thus to $\insZ$.

If now $\Omega$ is an open subset of $\Spec(G)$, then $\D(f)\subseteq\Omega$ for some $f\in G^+$. There is an $x\in X$ such that $f(x)>0$, and thus $f\notin P_x$, i.e., $P_x\in \D(f)$; hence $P_x\in\mathcal{P}\cap\D(f)\subseteq\mathcal{P}\cap\Omega$, which in particular is nonempty. Thus $\mathcal{P}$ is dense.
\end{proof}

\begin{defin}
Let $P_x$ be as in Proposition \ref{prop:finite}\ref{prop:finite:filter}. If $P_x\neq G$, we call $P_x$ the \emph{finite prime associated to $x$}, and we say that $P\in\Specast(G)$ is a \emph{finite prime} of $G$ if $P=P_x$ for some $x\in X$. 

If $P\neq G $ is not a finite prime, we say that $P$ is a \emph{prime at infinity} (or an \emph{infinite prime} of $G$). We denote by $\Spec^\infty(G)$ the set of primes at infinity of $G$.
\end{defin}

\begin{oss}\label{oss:separaz-spec}
Being a prime at infinity is not an inherent algebraic property of $P$, but rather it depends from its realization as a subgroup of $\funct(X,\insZ)$. For example, let $X$ be an infinite set and let $G$ be the group generated by all the $e_x$ and by the function $\mathbf{1}_X$ that is constantly equal to $1$; then, $G$ is a free group with basis $\{e_x,\mathbf{1}_X\mid x\in X\}$. The set $P:=\langle e_x\mid x\in X\rangle$ is a prime ideal of $G$, and is a prime at infinity of $G$ since $P\neq P_x$ for every $x\in X$.

Let now $\overline{x}$ be an element not in $X$, let $X':=X\cup\{\overline{x}\}$ and consider the map
\begin{equation*}
\begin{aligned}
\phi\colon G & \longrightarrow \funct(X',\insZ),\\
e_x& \longmapsto e_x,\\
\mathbf{1}_X& \longmapsto \mathbf{1}_{X'}.
\end{aligned}
\end{equation*}
Then, $\phi$ is an $\ell$-group isomorphism between $G$ and $\phi(G)$; however, the image of $P$ is a finite prime, namely $\phi(P)=P_{\overline{x}}$. Thus $\Spec^\infty(\phi(G))=\emptyset$.
\end{oss}

The following lemma is trivial, but we spell it explicitly for clarity.
\begin{lemma}
Let $G\subseteq\funct(X,\insZ)$ be an $\ell$-group, $f\in G^+$, let $x\in X$. Then, $x\in\supp(f)$ if and only if $P_x\in\D(f)$.
\end{lemma}

\section{dd-domains}
As explained in Section \ref{sect:prufer-ell}, $\ell$-groups are exactly the structures that arise as groups of invertible ideals of a Pr\"ufer domain. In this section, we study what happens when we restrict to subgroups of a group $\funct(X,\insZ)$ of integer-valued functions.

\begin{defin}
Let $D$ be a Pr\"ufer domain. We say that a subset $X\subseteq\Spec(D)$ is a \emph{discrete core} for $D$ if $X$ is dense in $\Spec(D)$ (with respect to the inverse topology) and $D_P$ is a DVR for every $P\in X$.

We say that $D$ is a \emph{discretely dense domain} (for short, a \emph{dd-domain}) if it has a discrete core.
\end{defin}

An equivalent definition is that $D$ is a dd-domain if the set of discrete prime ideals of $D$ is dense, with respect to the inverse topology. However, a dd-domain can have more than one discrete core (see Example \ref{ex:dd}\ref{ex:dd:ad2}), and the definition above allows a greater degree of flexibility.


Let $D$ be a dd-domain, and let $X$ be a discrete core of $D$. To every fractional ideal $I$ we can associate an \emph{ideal function} (relative to $X$) defined by
\begin{equation*}
\begin{aligned}
\nu_I\colon X &\longrightarrow \insZ,\\
P & \longmapsto v_P(I)
\end{aligned}
\end{equation*}
where $v_P$ is the valuation relative to $D_P$.

\begin{teor}\label{teor:dd-XZ}
Let $D$ be a Pr\"ufer domain.
\begin{enumerate}[(a)]
\item\label{teor:dd-XZ:dd->} If $D$ is a dd-domain and $X$ is a discrete core, the map 
\begin{equation*}
\begin{aligned}
\Phi\colon \Inv(D) &\longrightarrow \funct(X,\insZ),\\
I & \longmapsto \nu_I,
\end{aligned}
\end{equation*}
is an injective $\ell$-group homomorphism. In particular, $\Inv(D)$ is isomorphic (as an $\ell$-group) to a subgroup of $\funct(X,\insZ)$.
\item\label{teor:dd-XZ:->dd} Let $X$ be any set and $G\subseteq\funct(X,\insZ)$ be an $\ell$-group. If $\Inv(D)\simeq G$, then $D$ is a dd-domain.
\end{enumerate}
\end{teor}
\begin{proof}
\ref{teor:dd-XZ:dd->} Since $v_P$ is a valuation, for every pair of fractional ideals $I,J$ we have $v_P(IJ)=v_P(I)+v_P(J)$. Hence $\Phi$ is a group homomorphism. Moreover, the order structure of $\Inv(D)$ is the reverse containment, and thus $I\wedge J=I+J$. Hence,
\begin{equation*}
v_P(I\wedge J)=v_P(I+J)=\min(v_P(I),v_P(J))
\end{equation*}
and so $\Phi(I\wedge J)=\Phi(I)\wedge\Phi(J)$. Hence $\Phi$ is a homomorphism of $\ell$-groups.

We show that it is injective: suppose that $\Phi(I)$ is the zero map. Then, $\nu_I(P)=v_P(I)=0$ for all $P\in X$. Thus,
\begin{equation*}
I\subseteq\bigcap_{P\in X}ID_P=\bigcap_{P\in X}D_P=D,
\end{equation*}
with the last equality following by the density of $X$ \cite[Theorem 5.6(a)]{olberding_affineschemes} (note that since $D$ is Pr\"ufer the localization map establishes a homeomorphism between the spectrum and Zariski space of $D$, in both the Zariski and the inverse topology \cite[Lemma 2.4]{dobbs_fedder_fontana}). Moreover, $\V(I)\cap X=\emptyset$. Since $I$ is invertible, it is finitely generated, and thus $\V(I)$ is an open set of $\Spec(D)$; since $X$ is dense, we must have $\V(I)=\emptyset$, i.e., $I=D$. Thus $\Phi$ is injective.

\ref{teor:dd-XZ:->dd} If $P_x$ is a finite prime of $G$, then by Proposition \ref{prop:finite}\ref{prop:finite:residue} the residue group $G/P_x$ is isomorphic to $\insZ$. Thus, the corresponding prime ideal $Q_x$ of $D$ is such that the localization $D_{Q_x}$ is a DVR.

Moreover, by Proposition \ref{prop:finite}\ref{prop:finite:dense} $\mathcal{P}:=\{P_x\mid x\in X\}$ is dense in $\Spec(G)$, and thus the corresponding set of prime ideals of $D$ is dense in $\Spec(D)$. Thus $\mathcal{P}$ is a discrete core, and $D$ is a dd-domain.
\end{proof}

The result of the previous theorem can also be expressed in a different way:
\begin{cor}
Let $D$ be a Pr\"ufer domain. Then, $D$ is a dd-domain if and only if $\Inv(D)$ is isomorphic as an $\ell$-group to a subgroup of $\funct(X,\insZ)$ for some nonempty set $X$.
\end{cor}

For dd-domains there is a strong connection between discrete cores and finite primes, as the next proposition shows.
\begin{prop}
Let $D$ be a dd-domain with discrete core $X$, and let $\Phi:\Inv(D)\longrightarrow G\subseteq\funct(X,\insZ)$ be the canonical isomorphism; let $\Psi:\Spec(D)\longrightarrow\Spec(G)$ be the corresponding map of spectra. Then, $P\in X$ if and only if $\Psi(P)$ is a finite prime of $G$.
\end{prop}
\begin{proof}
If $P\in X$, then the finite prime at $P$ is just the set of all invertible ideal $I$ such that $\nu_I(P)=0$, that is, the set of $I$ such that $ID_P=D_P$. Thus $\Psi(P)$ is just the finite prime at $P$. Since these are all the finite primes of $G$, if $P\in\Spec(D)\setminus X$ then $\Psi(P)$ is an infinite prime.
\end{proof}

Under this interpretation, Remark \ref{oss:separaz-spec} can be seen as looking at the group of invertible ideals of the same dd-domain $D$ through two different discrete cores, $X$ and $X'$.

We say that a subgroup $G\subseteq\funct(X,\insZ)$ is \emph{of bounded ratio} if, whenever $f,g\in G^+$ have the same support, then $nf\geq g$ for some $n\inN$.
\begin{cor}\label{cor:Phi-boundratio}
Let $D$ be a one-dimensional dd-domain and let $X$ be a discrete core of $D$. Let $\Phi$ be the map defined in the statement of Theorem \ref{teor:dd-XZ}\ref{teor:dd-XZ:dd->}. Then, $\Phi(\Inv(D))$ is a group of bounded ratio.
\end{cor}
\begin{proof}
Let $f=\Phi(I)$, $g=\Phi(J)$ be elements of $\Phi(\Inv(D))^+$ with the same support. Then, $\V(I)\cap X=\V(J)\cap X$. Since $D$ is one-dimensional, both $\V(I),\V(J)$ are clopen sets of $\Max(D)$; since $X$ is dense, it follows that $\V(I)=\V(J)$, and thus $I$ and $J$ have the same radical. Since $I$ is finitely generated, there is an integer $n$ such that $I^n\subseteq J$, and thus
\begin{equation*}
nf=n\Phi(I)=\Phi(I^n)\geq\Phi(J)=g.
\end{equation*}
Hence $\Phi(\Inv(D))$ is a group of bounded ratio.
\end{proof}

\begin{cor}\label{cor:1dim->br}
Let $X$ be a nonempty set and $G\subseteq\funct(X,\insZ)$ be an $\ell$-group. If $G$ is one-dimensional, then it is a group of bounded ratio.
\end{cor}
\begin{proof}
Let $D$ be a B\'ezout domain such that $\Inv(D)\simeq G$. Then, $D$ is a dd-domain (Theorem \ref{teor:dd-XZ}\ref{teor:dd-XZ:dd->}) and $\dim(G)=\dim(\Inv(D))=\dim(D)$. Thus $D$ is one-dimensional and $\Phi(\Inv(D))\simeq G$ is a group of bounded ratio by Corollary \ref{cor:Phi-boundratio}.
\end{proof}

\begin{ex}\label{ex:dd}
~\begin{enumerate}[(1)]
\item Every Dedekind domain is a dd-domain. Set $X:=\Max(D)$: then every maximal ideal $M$ is invertible, and the ideal function $\nu_M$ is just the characteristic function of $\{M\}$. Since every invertible ideal is a product of maximal ideals (possibly with negative exponents) the image $\Phi(\Inv(D))$ is exactly the direct sum $\dirsum(X,\insZ)$.
\item Every almost Dedekind domain is a dd-domain and, like in the Dedekind case, a discrete core is $X:=\Max(D)$. If $D$ is not a Dedekind domain, some maximal ideals are not invertible (and thus not all $e_x$ are in $\Phi(\Inv(D))$); conversely, some ideals $I$ are contained in infinitely many maximal ideals, and thus $\Phi(\Inv(D))$ is not contained in $\dirsum(X,\insZ)$.
\item\label{ex:dd:ad2} If $D$ is an almost Dedekind domain that is not Dedekind, there are discrete cores $X$ that are properly contained in $\Max(D)$. Indeed, since $D$ is not Dedekind that are always maximal ideals that are not invertible (equivalently, that are not finitely generated); if $M$ is one of them, then $M$ is not an isolated point of $\Max(D)$, and thus $\Max(D)\setminus\{M\}$ is a discrete core of $D$. A less obvious discrete core is the set of all maximal ideals that are not critical (see \cite[Theorem 5.6]{bounded-almded}). More generally, every Hausdorff zero-dimensional compact space can be realized as the maximal space of an almost Dedekind domain \cite[Section 3]{olberding-factoring-SP}, and any dense subset of such a space can be used as a discrete core.
\item There are one-dimensional Pr\"ufer domain that are dd-domains but not almost Dedekind. For example, if $D$ has a unique noninvertible maximal ideal, say $N$, and $N$ is not isolated in $\Max(D)$, then $X:=\Max(D)\setminus\{N\}$ is a discrete core (since every other maximal ideal $M$ is invertible and thus the localization $D_M$ is a DVR); if, moreover, $D_N$ is not a DVR then $D$ is not almost Dedekind. See Example \ref{ex:limitQ} for an $\ell$-group arising from a domain of this kind.
\item Let $p$ be a prime number. If $A=\insZ_p$ is the localization of $\insZ$ at $(p)$, then the ring $\mathrm{Int}(A)$ of the integer-valued polynomials is a dd-domain; consequently, also $\mathrm{Int}(\insZ)$ is a dd-domain. On the other hand, if $B$ is the ring of the $p$-adic integers, then $\mathrm{Int}(B)$ is not a dd-domain as it is not the intersection of any family of one-dimensional localizations (and thus, in particular, $\mathrm{Int}(B)$ cannot have a discrete core) \cite[Proposition 3 and subsequent Corollaires]{chabert-fatou}.
\item The ring $E$ of all entire functions (i.e., of all functions $\insC\longrightarrow\insC$ that are holomorphic everywhere) is a dd-domain. Setting $M_\alpha$ to be the maximal ideal generated by $z-\alpha$, then $X:=\{M_\alpha\mid\alpha\inC\}$ is a discrete core (since every nonzero nonunit entire function must have zeros), and the range of the map $\Phi$ is the whole $\funct(X,\insZ)$. Note that $E$ is infinite-dimensional \cite{henriksen_prime}.
\end{enumerate}
\end{ex}

Let $G\subseteq\funct(X,\insZ)$ be an $\ell$-group. We can construct a B\'ezout domain $D$ such that $\Inv(D)\simeq G$, and by Theorem \ref{teor:dd-XZ} $D$ will be a dd-domain; using again Theorem \ref{teor:dd-XZ} we can construct an isomorphism $\Inv(D)\longrightarrow G'\subseteq\funct(X',\insZ)$, where $X'$ is a discrete core of $D$. In general, even if $X'=X$, there is no guarantee that $G'=G$: for example, if $G=\langle 2e_x\mid x\in X\rangle$, then $D$ will be a Dedekind domain, and $\Phi(\Inv(D))$ will just be the direct sum $\dirsum(X,\insZ)=\langle e_x\mid x\in X\rangle$. More generally, it is not hard to see that $G$ will never be the image of any function $\Phi$ from $\Inv(D)$ to $\funct(X,\insZ)$, for every Pr\"ufer domain $D$.

This fact does not cause any difficulty, but requires a technical workaround.
\begin{defin}
Let $G\subseteq\funct(X,\insZ)$ be an $\ell$-group. We say that $G$ is \emph{reduced} if there is a Pr\"ufer domain $D$ with discrete core $Y$ such that $G$ is the image of the composition
\begin{equation*}
\Inv(D)\xrightarrow{~~~\Phi~~~}\funct(Y,\insZ)\xrightarrow{~~~f^\ast~~~}\funct(X,\insZ)
\end{equation*}
where $\Phi$ is the map defined in Theorem \ref{teor:dd-XZ}\ref{teor:dd-XZ:->dd} and $f^\ast$ is the map induced by a bijection $f:Y\longrightarrow X$.
\end{defin}

\begin{prop}\label{prop:reduction-reduced}
Let $G\subseteq\funct(X,\insZ)$ be an $\ell$-group. There is an $X'$ and a reduced $\ell$-group $G'$ such that $G\simeq G'$.
\end{prop}
\begin{proof}
Let $D$ be a Pr\"ufer domain such that $\Inv(D)\simeq G$. By Theorem \ref{teor:dd-XZ}, $D$ is a dd-domain, and thus it has a discrete core $X'$. Applying again the theorem we see that $\Inv(D)$ is isomorphic to a subgroup of $\funct(X',\insZ)$, which is reduced by definition.
\end{proof}

The main argument of the proofs of the main results in Sections \ref{sect:successor} and \ref{sect:limit} will be done group-theoretically; however, when dealing with the spectrum, it is often easier to work with B\'ezout domains. We give two examples of this method of proof.

\begin{lemma}\label{lemma:supp-primeinfty}
Let $G\subseteq\funct(X,\insZ)$ be a one-dimensional $\ell$-group, and let $f\in G$. If $\supp(f)$ is finite, then $f$ belongs to every $P\in\Spec^\infty(G)$.
\end{lemma}
\begin{proof}
Let $D$ be a B\'ezout with an isomorphism $\Phi:\Inv(D)\longrightarrow G$, and let $\Psi:\Spec(D)\longrightarrow\Spec(G)$ the canonical homeomorphism. Then,
\begin{equation*}
\supp(f)=X\setminus\{x\in X\mid f(x)=0\}=X\setminus\{x\in X\mid f\in P_x\}.
\end{equation*}
Therefore, $f$ belongs to almost all $P_x$. Let $I:=\Phi^{-1}(f)$: since $\Psi(\V(I))=\D(f)$, we have that $I$ belongs to only finitely many prime ideals of $X':=\{\Psi(P_x)\mid x\in X\}$. Since $X'$ is dense in $\Max(D)$, it follows that $\V(I)\subseteq X'$, i.e., $I\nsubseteq Q$ for all $Q\in\Max(D)\setminus X'$. Hence $\Psi(Q)\notin\Psi(\V(I))=\D(f)$, i.e., $f\in\Psi(Q)$. Since $\Max(D)\setminus X'$ maps to $\Spec^\infty(G)$, the claim is proved.
\end{proof}

\begin{prop}
Let $G\subseteq\funct(X,\insZ)$ be a reduced $\ell$-group. Then:
\begin{enumerate}[(a)]
\item $P_x\neq P_y$ for every $x\neq y$;
\item for every $x\in X$ there is an $f_x\in G^+$ such that $f_x(x)=1$.
\end{enumerate}
\end{prop}
\begin{proof}
Let $D$ be a B\'ezout domain with discrete core $Y$ such that $G$ is the image of the composition
\begin{equation*}
\Inv(D)\xrightarrow{~~~\Phi~~~}\funct(Y,\insZ)\xrightarrow{~~~f^\ast~~~}\funct(X,\insZ).
\end{equation*}
We can suppose without loss of generality that $Y=X$ (and $f$ is the identity).

If $x\in X$, then $P_x$ is the image of the prime ideal $x$ of $D$; in particular, if $x\neq y$ then $P_x$ and $P_y$ are images of distinct primes, and thus are distinct. Moreover, there is a $d\in D$ such that $d\in x\setminus x^2$; if $f\in\funct(X,\insZ)$ is the image of the ideal $(d)$, then $f\in G^+$ and $f(x)=1$.
\end{proof}

\section{Derived sets}
Let $T$ be a topological space. The \emph{derived set} $\deriv(T)=\deriv^1(T)$ of $T$ is the set of all limit points of $T$ (i.e., of all the points that are not isolated); more generally, for every ordinal number $\alpha$ we can define the $\alpha$-th derived set of $T$ by
\begin{equation*}
\deriv^\alpha(T):=\begin{cases}
T & \text{if~}\alpha=0,\\
\deriv(\deriv^\beta(T)) & \text{if~}\alpha=\beta+1,\\
\bigcap_{\beta<\alpha}\deriv^\beta(T) & \text{if~}\alpha\text{~is a limit ordinal}.
\end{cases}
\end{equation*}
The \emph{Cantor-Bendixson rank} $\cb(p)$ of a point $p\in T$ is the largest $\alpha$ such that $p\in\deriv^\alpha(T)$, or $\infty$ if such an $\alpha$ does not exist. The \emph{Cantor-Bendixson rank} $\cb(T)$ of $T$ is the smallest ordinal $\alpha$ such that $\deriv^\alpha(T)=\deriv^{\alpha+1}(T)$, and in this case we set $\deriv^\infty(T):=\deriv^\alpha(T)=\deriv^{\alpha+1}(T)$ (this set is called the \emph{perfect hull} of $T$). Therefore,
\begin{equation*}
\deriv^\alpha(T)=\{p\in T\mid \cb(p)\geq\alpha\}
\end{equation*}
and 
\begin{equation*}
\cb(T)=\sup\{\cb(p)\mid p\in T\}.
\end{equation*}
If $T$ is compact, the supremum above is a maximum.

We say that $T$ is \emph{scattered} if $\deriv^\infty(T)=\emptyset$.

We shall be interested in the Cantor-Bendixson rank of the finite primes $P_x$. We will use the following notation.
\begin{defin}
Let $G\subseteq\funct(X,\insZ)$ be a reduced $\ell$-group.
\begin{itemize}
\item If $P\in\Specast(G)$, we denote by $\cb(P)$ the Cantor-Bendixson rank of $P$ as a point of $\Specast(G)$.
\item If $x\in X$, we set $\cb(x):=\cb(P_x)$.
\item If $f\in G^+$, we set $\cb(f):=\sup\{\cb(P)\mid P\in\D(f)\}$.
\item If $f\in G$, we set $\cb(f)=\max(\cb(f^+),\cb(f^-))$, where $f^+:=f\vee 0$ and $f^-:=f\wedge 0$ (and $\cb(0)=0$). 
\item For every ordinal number $\alpha$ and for $\alpha=\infty$, we set
\begin{equation*}
\derivX^\alpha(X):=\{x\in X\mid P_x\in\deriv^\alpha(\Spec(G))\cap X\}.
\end{equation*}\end{itemize}
\end{defin}

We note that the notation $\derivX^\alpha(X)$ can be ambiguous, as its definition depends on $G$, which is not reflected in the notation. However, we shall never deal at the same time with derived sets induced by different groups, and thus we avoid specifying $G$ to streamline the notation.

\begin{lemma}\label{lemma:sum-cb}
Let $G\subseteq\funct(X,\insZ)$ be a reduced $\ell$-group, and let $f,g\in G^+$. Then, $\cb(f+g)=\max\{\cb(f),\cb(g)\}$.
\end{lemma}
\begin{proof}
Since $f,g\in G^+$, we have $\D(f+g)=\D(f)\cup\D(g)$. Indeed, suppose that $P\in\D(f)\cup\D(g)$: then, without loss of generality, $P\in\D(f)$, i.e., $f\notin P$. Since $0\leq f\leq f+g$, the convexity of $P$ implies that $f+g\notin P$, i.e., $P\in\D(f+g)$. Conversely, if $P\in\D(f+g)$, then $f+g\notin P$ and thus at least one of $f$ and $g$ must be out of $P$ (since $P$ is a subgroup). Thus $P\in\D(f)\cup\D(g)$. The claim now follows from the definition of the Cantor-Bendixson rank.
\end{proof}

Suppose that $x$ is an isolated point of $X$. Then, $\{x\}$ is open, and thus $\{x\}=\supp(f)$, or equivalently $\{P_x\}=\D(f)$, for some $f\in G$. If $G$ is reduced, it follows that $G$ contains the basic element $e_x$. To generalize this fact to non-isolated elements, we introduce the following notion.
\begin{defin}
Let $G\subseteq\funct(X,\insZ)$ be an $\ell$-group and let $x\in X$. We say that $q\in G^+$ is a \emph{semibasic element at $x$} if $q(x)=1$ and $\supp(q)\cap\derivX^{\cb(x)}(X)=\{x\}$.
\end{defin}

\begin{lemma}\label{lemma:quark}
Let $G\subseteq\funct(X,\insZ)$ be a reduced $\ell$-group, and let $x\in X$.
\begin{enumerate}[(a)]
\item\label{lemma:quark:magg} If $q$ is semibasic at $x$, then $\supp(q)\cap\derivX^{\cb(x)+1}(X)=\emptyset$.
\item\label{lemma:quark:exists} If $x\notin\derivX^\infty(X)$, then there is a $q\in G^+$ that is semibasic at $x$.
\item\label{lemma:quark:isolated} If $x\notin\derivX^0(X)$, then $q=e_x$ is the unique element that is semibasic at $x$.
\end{enumerate}
\end{lemma}
\begin{proof}
\ref{lemma:quark:magg} Since $\derivX^{\cb(x)+1}(X)\subseteq\derivX^{\cb(x)}(X)$, we have
\begin{equation*}
\supp(q)\cap\derivX^{\cb(x)+1}(X)\subseteq\supp(q)\cap\derivX^{\cb(x)}(X)=\{x\}.
\end{equation*}
Since $x\notin\derivX^{\cb(x)+1}(X)$, we must have $\supp(q)\cap\derivX^{\cb(x)+1}(X)=\emptyset$.

\ref{lemma:quark:exists} Since $x\notin\derivX^\infty(X)$, the prime ideal $P_x$ is an isolated point of $\deriv^{\cb(x)}(\Spec(G))$, and thus there is an open set $\Omega$ of $\Spec(G)$ such that $\Omega\cap\deriv^{\cb(x)}(\Spec(G))=\{P_x\}$. Therefore, we can find $f\in G^+$ such that $P_x\in\supp(f)\subseteq\Omega$; for such a $f$, we have $\supp(f)\cap\derivX^{\cb(x)}(X)=\{x\}$. Since $G$ is reduced, we can find a $g\in G^+$ such that $g(x)=1$; therefore, $q:=f\wedge g$ is semibasic at $x$. 

\ref{lemma:quark:isolated} If $x\notin\derivX^0(X)$, then also $x\notin\derivX^\infty(X)$ and thus, by the previous point, there is a semibasic element $q$; in particular, $\{x\}=\supp(q)\cap\derivX^0(X)=\supp(f)\cap X=\supp(q)$ and thus $q=ne_x$ for some $n\inN^+$. Since $q$ is semibasic, $1=q(x)=ne_x(x)=n$ and so $q=e_x$ is the unique element that is semibasic at $x$.
\end{proof}

\begin{oss}\label{oss:cantor}
Let $G$ be a reduced $\ell$-group. Then, the map $\phi:X\longrightarrow\Spec(G)$, $x\mapsto P_x$ is injective, and thus it can be used to endow $X$ with a topology (depending on $G$). It is easy to see that, under this topology, if $x\in X$ is a limit point then $P_x$ is a limit point of $\Spec(G)$, and thus
\begin{equation*}
\phi(\deriv(X))\subseteq\deriv(\Spec(G))\cap\phi(X)=\derivX(\Spec(G)).
\end{equation*}
An inductive argument shows that $\phi(\deriv^\alpha(X))\subseteq\derivX^\alpha(\Spec(G))$ for all $\alpha$. However, the two sets need not to coincide; to show this, we provide an example of a topological space $T$ with a dense set $X$ such that $\deriv^\infty(X)\neq\deriv^\infty(T)\cap X$.

Let $Y$ be the Cantor space in $[0,1]$, and let $Y':=\{(p,0)\mid p\in Y\}\subseteq\insR^2$. For every integer $n\geq 1$, let $X_n$ be the set of all $(x,1/n)\in\insR^2$ such that $x\in[0,1]$ and the $k$-th digit of $x$ in base $3$ is $0$ for every $k\geq n$ (we represent $1$ as $0,2\cdots 2\cdots$). Note that each $X_n$ is a finite set. Let
\begin{equation*}
X:=\{(0,0)\}\cup\bigcup_{n\geq 1}X_n,\quad T=X\cup Y'.
\end{equation*}
By construction, $X$ is dense in $T$ (if the expansion in base $3$ of $p$ is $p=0,c_1c_2\cdots c_n\cdots$, then $(p,0)$ is the limit of $\{(p_k,1/k)\}$, where $p_k=0,c_1c_2\cdots c_k$ in base $3$).

Each element of $X\setminus\{(0,0)\}=T\setminus Y'$ is discrete in $T$, while every point of $Y'$ is a limit point; hence,
\begin{equation*}
Y'=\deriv(T)=\deriv^2(T)=\cdots=\deriv^\infty(T).
\end{equation*}
In particular, $(0,0)\in\deriv^\infty(T)\cap X$. However, $(0,0)$ is the unique limit point of $X$, and thus
\begin{equation*}
\deriv(X)=\{(0,0)\},\quad\deriv^2(X)=\emptyset.
\end{equation*}
Hence $\deriv^\alpha(X)\neq\deriv^\alpha(T)\cap X$ for every $\alpha\geq 2$ and for $\alpha=\infty$.
\end{oss}

We end this section with a lemma that will be useful later.
\begin{lemma}\label{lemma:quark-linindip}
Let $G\subseteq\funct(X,\insZ)$ be an $\ell$-group, and let $Y\subseteq X$ be a nonempty subset. For all $y\in Y$, let $q_y$ be semibasic at $y$. Then, $Q:=\{q_y\mid y\in Y\}$ is a linearly independent subset of $G$.
\end{lemma}
\begin{proof}
Suppose that $n_1q_{y_1}+\cdots+n_kq_{y_k}=0$, where $y_1,\ldots,y_k\in Y$ (with $y_i\neq y_j$ if $i\neq j$) and $n_1,\ldots,n_k\inZ$. Let $\gamma:=\sup\{\cb(y_i)\mid i=1,\ldots,k\}$, and suppose without loss of generality that $\cb(y_1)=\gamma$. Then, $q_{y_i}(y_1)=0$ for all $i>1$, and thus 
\begin{equation*}
0=(n_1q_{y_1}+\cdots+n_kq_{y_k})(y_1)=n_1q_{y_1}(y_1)=n_1,
\end{equation*}
so that $n_1=0$. Repeating the process we obtain that $n_k=0$ for all $k$, and thus $Q$ is a linearly independent set.
\end{proof}

\section{Basic limit groups}
Let $G\subseteq\funct(X,\insZ)$ be an $\ell$-group. The set $\Spec^\infty(G)$ of the infinite primes can be very large and thus very difficult to analyze. For this reason, the main focus of this paper will be in cases where this space is small enough to be controlled: the main idea is to use the derived sequence to ``peel off'' the layers of isolated points, and this method will require that $\Spec^\infty(G)$ is countable and disjoint from the perfect hull of $\Spec(G)$.

When $\Spec^\infty(G)$ is empty, we can prove freeness by appealing to the ring case.
\begin{prop}
Let $G\subseteq\funct(X,\insZ)$ be an $\ell$-group. If $\Spec^\infty(G)=\emptyset$, then $G$ is free.
\end{prop}
\begin{proof}
Let $D$ be a B\'ezout domain such that $\Inv(D)\simeq G$. If $\Spec^\infty(G)$ is empty, then all residue groups of $G$ are isomorphic to $\insZ$, and thus all value groups of $D$ are discrete; it follows that $D$ is an almost Dedekind domain. By \cite[Proposition 5.3]{bounded-almded}, $\Inv(D)$ is free.
\end{proof}

Thus, we only need to consider the case in which $\Spec^\infty(G)$ is nonempty; in the following, we shall always silently assume this hypothesis. We will also only consider one-dimensional groups: in this case, $\Specast(G)=\Spec(G)\setminus\{G\}$ is the set of minimal elements of $\Spec(G)$, and since $G$ is a generic point of $\Spec(G)$ (i.e., it is contained in every nonempty open set) a subset of $\Specast(G)$ is dense in $\Specast(G)$ if and only if it is dense in $\Spec(G)$. Furthermore, in this case $\Specast(G)$ is a zero-dimensional Hausdorff space.

The next step is to consider $\ell$-groups which have a single prime at infinity; this case will be the focus of this and the next two sections. To simplify the discussion, we introduce the following definition.
\begin{defin}
Let $X$ be a nonempty set and $G\subseteq\funct(X,\insZ)$ be an $\ell$-group. We say that $G$ is a \emph{basic limit group} if $\Spec^\infty(G)$ is a singleton. 

In this case, if $P$ is its prime at infinity, we call the residue group $G/P$ the \emph{residue group at infinity} of $G$, and we denote it by $\residuo_\infty(G)$. We denote by $\pi_\infty$ the canonical surjective homomorphism
\begin{equation*}
\pi_\infty\colon G\longrightarrow\residuo_\infty(G).
\end{equation*}
\end{defin}

\begin{oss}\label{oss:basic}
~\begin{enumerate}[(1)]
\item Remark \ref{oss:separaz-spec} shows that being a basic limit group is not an abstract property of $G$, but it depends by its realization inside $\funct(X,\insZ)$.
\item\label{oss:basic:deriv} Let $G$ be a basic limit group with infinite prime $P$, and let $F$ be the set of finite primes of $G$. Then, $P$ is a limit point of $\Specast(G)$ and so $\deriv(\Specast(G))=\deriv(F)\cup\{P\}$; applying again the derived set construction we obtain that $\deriv^2(\Specast(G))$ will be equal to $\deriv^2(F)$ or to $\deriv^2(F)\cup\{P\}$. By induction, it follows that $\deriv^\alpha(\Specast(G))$ will be equal to $\deriv^\alpha(F)$ or to $\deriv^\alpha(F)\cup\{P\}$. In particular, the phenomenon shown in Remark \ref{oss:cantor} cannot happen: we always have $\phi(\deriv^\alpha(X))=\deriv^\alpha(\Specast(G))\cap\phi(X)=\derivX^\alpha(X)$. Therefore, for every $A\subseteq X$, we have $\cb(A)=\cb(\phi(A))$.
\end{enumerate}
\end{oss}

The aim of this and of the following two sections is to show that a basic limit group is free under some hypothesis on its spectrum and on its residue group at infinity. For this reason, we will distinguish two cases, according to whether the Cantor-Bendixson rank of the prime at infinity is a successor or a limit ordinal, which we consider respectively in Section \ref{sect:successor} and in Section \ref{sect:limit}.

In this section, we prove a few results that allow to impose further restrictions on $\Specast(G)$ and $P$ without losing generality.

\begin{lemma}\label{lemma:restriction-freekernel}
Let $G\subseteq\funct(X,\insZ)$ be a reduced basic limit group with prime at infinity $P$. Let $C$ be a closed subset of $\Specast(G)$ containing $P$. Then, the restriction map $\phi:G\longrightarrow\funct(C,\insZ)$ has free kernel.
\end{lemma}
\begin{proof}
Let $D$ be a B\'ezout domain such that $\Inv(D)\simeq G$, and let $C'\subseteq\Max(D)$ the closed set corresponding to $C$. Let $T:=\bigcap\{D_Q\mid Q\in C'\}$. If $A\in\Max(D)$ is a prime ideal such that $AT=T$, then $A$ correspond to a finite prime of $G$, and thus its value group is $\insZ$; hence, $PT=T$. By \cite[Theorem 4.10]{overring-operators}, it follows that the kernel of $\Inv(D)\longrightarrow\Inv(T)$ is free. However, this kernel corresponds to the kernel of $\phi$ throught the commutative diagram
\begin{equation*}
\begin{CD}
G @>{\phi}>> \funct(C,\insZ)\\
@VVV    @VVV\\
\Inv(D) @>>> \Inv(T).
\end{CD}
\end{equation*}
Therefore, $\ker\phi$ is free.
\end{proof}

A topological space $X$ is \emph{Fr\'echet-Urysohn} is the closure of every subspace is equal to its sequential closure, i.e., if $S\subseteq X$ and $s$ belongs to the closure of $S$ then there is a sequence $\{s_n\}_{n=1}^\infty\subseteq S$ such that $s$ is a limit of $\{s_n\}_{n=1}^\infty$. This property is very general; for example, every first-countable space is Fr\'echet-Urysohn.
\begin{prop}\label{prop:reduction-maxcb}
Let $G\subseteq\funct(X,\insZ)$ be a reduced one-dimensional basic limit group with prime at infinity $P\notin\deriv^\infty(\Specast(G))$. Then, we can find a set $X'$ and a reduced basic limit group $G'\subseteq\funct(X',\insZ)$ with prime at infinity $P'$ such that:
\begin{itemize}
\item $\cb(P')=\cb(\Specast(G'))$;
\item $\cb(P'')<\cb(\Specast(G')$ for every prime $P''\neq P'$ of $G'$;
\item there is a surjective map $\pi:G\longrightarrow G'$ with free kernel.
\end{itemize}
Moreover, if $\cb(P)$ is a successor ordinal, we can choose $G'$ such that $\cb(P')=1$; if also $\Specast(G)$ is Fr\'echet-Urysohn, we can take $X'$ to be countable.
\end{prop}
\begin{proof}
Since $P\notin\deriv^\infty(\Specast(G))$, there is an open subset $O$ of $\Specast(G)$ such that $O\cap\deriv^{\cb(P)}(\Specast(G))=\{P\}$; since $G$ is one-dimensional, we can also suppose that $O$ is clopen. The claim now follows taking $X':=\{x\in X\mid P_x\in O\}$, setting $G'$ to be the image of the restriction map $G\longrightarrow\funct(X',G)$ and applying Lemma \ref{lemma:restriction-freekernel}.

If $\cb(P)=\beta+1$ is a successor ordinal, let $X'$ be as above and take $X'':=X'\cap\deriv^\beta(X)$. If also $\Specast(G)$ is Fr\'echet-Urysohn, there is a countable sequence $\{P_{x_n}\}$ of finite primes with limit $P$; thus we take $X''':=X''\cap\{x_n\mid n\inN\}$. In both cases another application of Lemma \ref{lemma:restriction-freekernel} gives the desired conclusion.
\end{proof}

\begin{lemma}\label{lemma:spanqx}
Let $G\subseteq\funct(X,\insZ)$ be a one-dimensional basic limit group, and let $\beta$ be an ordinal number. For every $x\in\derivX^\beta(X)$, let $q_x$ be semibasic at $x$. If $f\in \ker\pi_\infty$ and $\cb(f)\leq\beta$, then $f\in\langle q_x\mid x\in\derivX^\beta(X)\rangle$.
\end{lemma}
\begin{proof}
We note that, by Lemma \ref{lemma:quark}\ref{lemma:quark:exists}, all such elements $q_x$ exist.

Since $f\in\ker\pi_\infty$, then $\D(f)=\{P_x\mid x\in\supp(f)\}$; in particular, $\supp(f)\simeq\D(f)$ is compact. We proceed by induction on $\beta$.

If $\beta=0$, then $\supp(f)$ must be discrete and thus finite, say $\supp(f)=\{x_1,\ldots,x_n\}$. Then, $q_{x_i}=e_{x_i}$ for all $i$ and $f=f(x_1)e_{x_1}+\cdots+f(x_n)e_{x_n}$ is the desired decomposition.

Suppose the claim holds for every $\alpha<\beta$. If $\beta$ is a limit ordinal, then $\cb(\supp(f))<\beta$ (since $\supp(f)$ is compact) and the claim holds by induction. Suppose that $\beta=\alpha+1$ is a limit ordinal. The space $\supp(f)\cap\derivX^\beta(X)$ is closed and compact in $\derivX^\beta(X)$, and thus it is finite, say equal to $x_1,\ldots,x_n$. For each $i$ and each $t\in\derivX^\beta(X)$, we have $q_{x_i}(t)=0$ unless $t=x_i$, in which case $q_{x_i}(x_i)=1$; thus, the function
\begin{equation*}
f_1:=f-\sum_{i=1}^nf(x_i)q_{x_i}\in G
\end{equation*}
is $0$ for every $t\in\derivX^\beta(X)$. By construction, $\pi_\infty(f_1)=\pi_\infty(f)=0$ and $f_1(t)=0$ for all $t\in\derivX^\beta(X)$; hence $\cb(f_1)<\beta$. By induction, $f_1$ can we written as a sum of semibasic elements $q_y$ with $\cb(y)\leq\alpha$; hence $f$ can be written as a sum of semibasic elements $q_y$ with $\cb(y)\leq\beta$. The claim is proved.
\end{proof}

\begin{cor}\label{cor:Specast-scattered-freeker}
Let $G\subseteq\funct(X,\insZ)$ be a one-dimensional basic limit group and suppose that $\Specast(G)$ is scattered. Then, $\ker\pi_\infty$ is free. In particular, if $\kappa_\infty(G)$ is free then $G$ is free.
\end{cor}
\begin{proof}
For every $x\in X$, let $q_x$ be semibasic at $x$; such elements exists by Lemma \ref{lemma:quark}\ref{lemma:quark:exists}, because $\Specast(G)$ is scattered. If $f\in\ker\pi_\infty$, then $f\in\langle q_x\mid x\in X\rangle$ by Lemma \ref{lemma:spanqx} (using $\beta=\cb(f)$); thus $\ker\pi_\infty$ is generated by $\mathbf{q}:=\{q_x\mid x\in X\}$. Moreover, $\mathbf{q}$ is linearly independent by Lemma \ref{lemma:quark-linindip}. Hence it is a basis of $\ker\pi_\infty$.

The ``in particular'' statement follows by considering the exact sequence
\begin{equation*}
0\longrightarrow\ker\pi_\infty\longrightarrow G\longrightarrow\kappa_\infty(G)\longrightarrow 0.
\end{equation*}
The claim is proved.
\end{proof}

\begin{cor}\label{cor:generatori-ex-A}
Let $G\subseteq\funct(X,\insZ)$ be a one-dimensional reduced basic $\ell$-group, and suppose that $\Specast(G)$ is scattered. For every $x\in X$, let $q_x\in G^+$ be a semibasic element at $x$, and let $A\subseteq G^+$ be such that $\pi_\infty(A)$ generates the residue group at infinity $\pi_\infty(G)$. Then,
\begin{equation*}
G=\langle A,q_x\mid x\in X\rangle.
\end{equation*}
\end{cor}
\begin{proof}
Let $G':=\langle A,q_x\mid x\in X\rangle$. By Corollary \ref{cor:Specast-scattered-freeker}, $\ker\pi_\infty=\langle q_x\mid x\in X\rangle\subseteq G'$. If now $f\notin\ker\pi_\infty$, then there are $a_1,\ldots,a_k\in A$, $\lambda_1,\ldots,\lambda_k\in\insZ$ such that $\pi_\infty(f)=\pi_\infty(\lambda_1a_1+\cdots+\lambda_ka_k)$; thus, $f_1:=f-\lambda_1a_1+\cdots+\lambda_ka_k$ satisfies $\pi_\infty(f_1)=0$, and by the previous part of the proof $f_1\in G'$. Since each $a_i$ is in $G'$, we have $f\in G'$.

It follows that $G'=G$, as claimed.
\end{proof}

Corollary \ref{cor:Specast-scattered-freeker} shows that if $\Specast(G)$ is scattered and the residue group at infinity is free, then $G$ is free; the purpose of the next two sections is to relax the condition on $\kappa_\infty(G)$ by allowing other classes of groups (like divisible groups). We shall distinguish three cases, according to whether the Cantor-Bendixson rank is a successor ordinal (Theorem \ref{teor:onelimitpoint}), a limit ordinal with countable cofinality (Theorem \ref{teor:limit-countcof}) or a limit ordinal with uncountable cofinality (Theorem \ref{teor:limit-uncountcof}).

The structure of the proofs of the three cases is essentially the same:
\begin{itemize}
\item we reduce to the case where $\cb(P)=\cb(\Specast(G))$ using Proposition \ref{prop:reduction-maxcb};
\item we construct a family $\{q_x\}_{x\in X}$ of semibasic elements and a family $\{f_{\lambda,n}\}$ whose image generates $\kappa_\infty(G)$; by Corollary \ref{cor:generatori-ex-A}, these two families together generate $G$;
\item we construct two sequences of groups, $\{A_\beta\}$ and $\{B_\beta\}$, such that, for every $\beta$:
\begin{itemize}
\item $A_\beta\subseteq B_\beta\subseteq A_{\beta+1}$;
\item $B_\beta/A_\beta$ is a group of bounded torsion;
\item $A_{\beta+1}/B_\beta$ is free (we find explictly a basis);
\item consequently, $B_{\beta+1}/B_\beta$ is free;
\item $\{B_\beta\}$ is smooth and has union $G$.
\end{itemize}
\end{itemize}
The last two points will prove that $G$ is free, as requested. The main difference between the three cases is the definition of the family $\{f_{\lambda,n}\}$; such a difference is needed to control the torsion of $B_\beta/A_\beta$.

As a last preliminary, we define the class of groups that will be allowed as residue groups at infinity.
\begin{defin}
Let $G$ a subgroup of $\insR$. A \emph{$\insQ$-basis} of $G$ is a $\insQ$-linearly independent set $\{b_\lambda\}_{\lambda\in\Lambda}$ such that
\begin{equation*}
G=\bigoplus_{\lambda\in\Lambda}(G\cap b_\lambda\insQ).
\end{equation*}
\end{defin}

\begin{oss}
~\begin{enumerate}
\item If $G$ is free, then any basis is linearly independent and thus a $\insQ$-basis.
\item If $G$ is divisible, it is a vector space over $\insQ$, and any basis of $G$ as a $\insQ$-vector space is a $\insQ$-basis.
\end{enumerate}
\end{oss}

\section{The successor ordinal case}\label{sect:successor}
This section is devoted to the case where the prime at infinity $P$ is a successor ordinal, and we will use Proposition \ref{prop:reduction-maxcb} to reduce to the case where $X=\insN$. Due to this, we can make use of the order structure of $\insN$ with the following notion.
\begin{defin}
Let $f\in\funct(\insN,\insZ)$. We denote by $\mu(f)$ the smallest integer $n$ such that $f(n)\neq 0$.
\end{defin}

\begin{defin}
Let $G\subseteq\funct(\insN,\insZ)$ be an $\ell$-group. We say that a sequence $\{a_n\}_{n=0}^\infty$ is a \emph{staircase base} for $G$ if the following conditions hold:
\begin{enumerate}
\item $a_n\in G^+$ for all $n$;
\item $\pi_\infty(\langle a_n\mid n\inN\rangle)=\residuo_\infty(G)$;
\item the sequence $\{\mu(a_n)\}_{n=0}^\infty$ is strictly increasing;
\item for every $n\inN$ there is a $t\inN^+$ such that $(ta_n-a_0)(k)=0$ for all $k\geq\mu(a_n)$.
\end{enumerate}
\end{defin}

\begin{ex}\label{ex:limitQ}
For every $n,k\inN$ let
\begin{equation*}
a_n(k):=\begin{cases}
0 & \text{if~}k<n,\\
k!/n!  & \text{if~} k\geq n.
\end{cases}
\end{equation*}
The first elements of $a_n$ for small $n$ are the following:
\begin{equation*}
\begin{matrix}
a_0: & 1 & 1 & 2 & 6 & 24 & \cdots\\
a_1: & 0 & 1 & 2 & 6 & 24 & \cdots\\
a_2: & 0 & 0 & 1 & 3 & 12 & \cdots\\
a_3: & 0 & 0 & 0 & 1 & 4 & \cdots
\end{matrix}
\end{equation*}
Then, $\{a_n\}$ is a staircase base for $G:=\langle a_n\mid n\inN\rangle$. Indeed, for every $n$ we have $\mu(a_n)=n$, and thus $\{\mu(a_n)\}_{n\inN}$ is strictly increasing. Moreover, $(n!a_n)(k)=k!=a_0(k)$ and thus $(n!a_n-a_0)(k)=0$ for all $k\geq n=\mu(a_n)$.

Furthermore, the staircase property guarantees that $e_n\in G$ and thus $\dirsum(\insN,\insZ)\subseteq G$. Then, $G$ is a basic limit group, since the set of the elements with finite support is prime ideal.

By construction, their image generates $\residuo_\infty(G)$.

Let $\pi:G\longrightarrow\insQ$ be the group homomorphism given by $\pi(a_n)=1/n!$. Then, $\pi$ is surjective and $\ker\pi=\dirsum(\insN,\insZ)$, that is, $\pi$ is the residue map of $G$. In particular, if $D$ is a B\'ezout domain such that $\Inv(D)\simeq G$, we have that every localization of $D$ except one is a DVR, while the remaining one has value group $\insQ$.
\end{ex}

The main reason we use staircase bases is to construct a set in $G$ that is linearly independent but generates a subgroup that is properly contains $\dirsum(\insN,\insZ)$. In the example above, this is guaranteed by the equality $e_n=n!a_n-a_0$ and the fact that $a_n\notin\dirsum(\insN,\insZ)$ for every $n$.

The following proposition shows that we can find a staircase base with good control over the relationship between the elements.

\begin{prop}\label{prop:staircase-exist}
Let $G\subseteq\funct(\insN,\insZ)$ be a reduced one-dimensional basic limit group. Suppose that every finite prime is isolated in $\Specast(G)$ and that $\residuo_\infty(G)$ is isomorphic to a subgroup of $\insQ$. Then, there are a staircase base $\{a_n\}_{n=0}^\infty\subseteq G^+$ and a sequence $\{d_n\}_{n=0}^\infty\subseteq\insN$ such that $d_n|n!$ and $(d_na_n-a_0)(k)=0$ for every $n\inN$ and every $k\geq\mu(a_n)$.
\end{prop}
\begin{proof}
Since $G$ is reduced and every finite prime is isolated, we have $e_x\in G$ for every $x\in \insN$. Thus, the kernel of the residue map is exactly the direct sum $\dirsum(\insN,\insZ)$.

Without loss of generality, we can consider $\pi_\infty$ as a map $G\longrightarrow\insQ$, so that $\residuo_\infty(G)\subseteq\insQ$. Take any $a_0\in G^+$ such that $\alpha_0:=\pi_\infty(a_0)\neq 0$, and, for every $n>0$, let $H_n:=\inv{n!}\alpha_0\insZ\cap\residuo_\infty(G)$. Then, $H_n$ is a subgroup of $\inv{n!}\alpha_0\insZ\simeq\insZ$, and thus it is itself isomorphic to $\insZ$; therefore, there is an $\alpha_n\in\residuo_\infty(G)^+$ such that $H_n=\alpha_n\insZ$, and we can find a $b_n\in G^+$ such that $\pi_\infty(b_n)=\alpha_n$. Moreover, $\bigcup_nH_n=\residuo_\infty(G)$, and thus $\residuo_\infty(G)=\langle\alpha_n\mid n\inN\rangle$.

By construction, $\alpha_0\in H_n$ and thus $\alpha_0=d_n\alpha_n$ for some $d_n\inN^+$ (both $\alpha_0$ and $\alpha_n$ are positive). Moreover, $n!\alpha_n\in n!H_n\subseteq\alpha_0\insZ$, and thus $n!\alpha_n\in d_n\alpha_n\insZ$; that is, $d_n|n!$. Hence,
\begin{equation*}
\pi_\infty(d_nb_n-a_0)=d_n\alpha_n-\alpha_0=0,
\end{equation*}
that is, $d_nb_n-a_0\in\ker\pi_\infty=\dirsum(\insN,\insZ)$. In particular, there is a $\lambda_n$ such that $(d_nb_n-a_0)(t)=0$ for all $t\geq\lambda_n$. For $n\geq 1$, we now define recursively
\begin{equation}\label{eq:defan-staircase}
a_n:=b_n-\sum_{i=0}^{\max\{\lambda_n,\mu(a_{n-1})\}}b_n(i)e_i
\end{equation}
and we claim that $\{a_n\}_{n=0}^\infty$ is a staircase base for $G$.

Since $b_n,e_i\in G$, we also have $a_n\in G$. Furthermore, by construction, the value in $k$ of the sum in \eqref{eq:defan-staircase} is either $b_n(k)$ or $0$, according to whether $k$ is larger or smaller than $\max\{\lambda_n,\mu(a_{n-1})\}$; it follows that $a_n(k)\geq 0$ for all $k$ (thus $a_n\in G^+$) and that $a_n(k)=0$ for $k\leq\max\{\lambda_n,\mu(a_{n-1})\}$. In particular, 
\begin{equation*}
\mu(a_n)>\max\{\lambda_n,\mu(a_{n-1})\}\geq\mu(a_{n-1}),
\end{equation*}
and so $\{\mu(a_n)\}_{n=0}^\infty$ is strictly increasing.

We have $a_n-b_n\in\dirsum(\insN,\insZ)=\ker\pi_\infty$, and thus $\pi_\infty(a_n)=\pi_\infty(b_n)=\alpha_n$. Hence, $\pi_\infty(\langle a_n\mid n\inN\rangle)=\langle\alpha_n\mid n\inN\rangle=\residuo_\infty(G)$.

Finally, suppose that $k\geq\mu(a_n)$. Then, $k>\max\{\lambda_n,\mu(a_{n-1})\}$, and thus $a_n(k)=b_n(k)$. It follows that
\begin{equation*}
(d_na_n-a_0)(k)=(d_nb_n-a_0)(k)=0
\end{equation*}
since in particular $k>\lambda_n$. Therefore, $\{a_n\}_{n=0}^\infty$ is a staircase base for $G$ satisfying the stated property.
\end{proof}

We shall need some elementary lemmas about the relationship between supports and primes at infinity. The first one is a partial converse to Lemma \ref{lemma:supp-primeinfty}.
\begin{lemma}\label{lemma:supp-infty}
Let $G$ be an $\ell$-group with $\dirsum(X,\insZ)\subseteq G$, and let $f\in G$. If $\supp(f)$ is infinite, then there is a $P\in\Spec^\infty(G)$ such that $f\notin G$.
\end{lemma}
\begin{proof}
Since $\dirsum(X,\insZ)\subseteq G$, every $P_x$ is an isolated point of $\Specast(G)$. By hypothesis, $\D(f)$ contains infinitely many finite primes, and thus in particular it is infinite; since it is also compact, not all of them can be isolated. Hence $\D(f)$ must meet $\Spec^\infty(G)$, as claimed.
\end{proof}

\begin{lemma}\label{lemma:cofinite}
Let $G$ be a reduced basic limit group such that every finite prime is isolated. Then, there is an infinite subset $Y\subseteq X$ such that $Y=\supp(f)$ for some $f\in G^+$ and, for every $g\in G^+$, one of $\supp(g)$ and $Y\setminus\supp(g)$ is finite.
\end{lemma}
\begin{proof}
Since $G$ is reduced and each $P_x$ is isolated, $e_x\in G$ for every $x\in X$. Let $P$ be the prime at infinity of $G$ and take any $f\in G^+\setminus P$: by Lemma \ref{lemma:supp-primeinfty}, $Y:=\supp(f)$ is infinite. Take $g\in G^+$, and suppose that $\supp(g)$ and $Y\setminus\supp(g)$ are both infinite. Let $g':=f-(f\wedge g)$: then, $g'\geq 0$ and $\supp(g')=Y\setminus\supp(g)$. By Lemma \ref{lemma:supp-infty}, there are primes at infinity $P,P'$ such that $g\notin P$ and $g'\notin P'$. Since $G$ is basic, $P=P'$ and $0=g\wedge g'\notin P$, a contradiction. Thus one of $\supp(g)$ and $Y\setminus\supp(g)$ must be finite.
\end{proof}

\begin{lemma}\label{lemma:free-boundtors}
Let
\begin{equation*}
0\longrightarrow A\longrightarrow B\longrightarrow C\longrightarrow 0
\end{equation*}
be an exact sequence of abelian groups. If $A$ is free, $B$ is torsion-free and $C$ is of bounded torsion then $B$ is free.
\end{lemma}
\begin{proof}
Let $n\in\insN^+$ be such that $nC=0$. Then, the map $n:B\longrightarrow B$, $b\mapsto nb$ is an injective group homomorphism (since $B$ is torsion-free) and its image $nB$ is contained inside $A$. Hence, $B\simeq nB$ is isomorphic to a subgroup of the free group $A$, and thus is free. 
\end{proof}

\begin{oss}
The previous lemma does not hold is $C$ is merely a torsion group: take for example $A=\insZ$ and $B=\insQ$.
\end{oss}

\begin{teor}\label{teor:onelimitpoint}
Let $G\subseteq\funct(X,\insZ)$ be a basic limit group of dimension one, and let $P$ be its prime at infinity. Suppose that:
\begin{itemize}
\item $\Specast(G)$ is Fr\'echet-Urysohn;
\item $P\notin\deriv^\infty(\Specast(G))$;
\item $\cb(P)$ is a successor ordinal;
\item $\kappa_\infty(G)$ has a $\insQ$-basis.
\end{itemize}
Then, $G$ is free.
\end{teor}
\begin{proof}
By Proposition \ref{prop:reduction-maxcb}, we can find a one-dimensional basic limit group $G'\subseteq\funct(X',\insZ)$ with $X'$ countable and a surjective map $\pi:G\longrightarrow G'$ such that $\ker\pi$ is free and $\pi(P)$ is an infinite prime of $G'$ with $\cb(\pi(P))=\cb(\Specast(G'))=1$. Thus, to prove that $G$ is free we can suppose without loss of generality that $X=\insN$ is countable and that all finite primes are isolated; in particular, $e_n\in G$ for every $n$. Further restricting to some $\supp(f)$ that belongs to the prime at infinity, by Lemma \ref{lemma:cofinite} we can also suppose that $\supp(f)$ is either finite or cofinite for every $f\in G$.

Let $\residuo_\infty(G)\subseteq\insR$ be the residue group at infinity of $G$. Let $\{b_\lambda\}_{\lambda\in\Lambda}$ be a $\insQ$-basis of $G$. For every $\lambda\in\Lambda$, let
\begin{equation*}
G_\lambda:=\pi_\infty^{-1}(b_\lambda\insQ\cap\residuo_\infty(G)).
\end{equation*}
Then, $G_\lambda$ is a subgroup of $G$ containing $\ker\pi_\infty$, and thus in particular it contains $\dirsum(\insN,\insZ)$. Moreover, it is a basic limit group with prime at infinity equal to $P\cap G_\lambda$. Hence, its residue group at infinity is exactly $b_\lambda\insQ\cap\residuo_\infty(G)\subseteq b_\lambda\insQ\simeq\insQ$. By Proposition \ref{prop:staircase-exist}, we can find a staircase base $\{a_{\lambda,n}\}_{n=0}^\infty$ of $G_\lambda$ such that $(d_{\lambda,n}a_{\lambda,n}-a_{\lambda,0})(k)=0$ for all $k\geq\mu(a_{\lambda,n})$, where $d_{\lambda,n}|n!$. In particular, $d_{\lambda,n}a_{\lambda,n}-a_{\lambda,0}\in\dirsum(\insN,\insZ)$ for every $n,\lambda$.

\medskip

We now define two sequences $\{A_r\}_{r=0}^\infty$ and $\{B_r\}_{r=0}^\infty$ of subgroups of $G$ recursively in the following way:
\begin{itemize}
\item $A_0=\langle e_0,a_{\lambda,0}\mid \mu(a_{\lambda,0})=0\rangle$;
\item $B_0=A_0$;
\item if $r\geq 1$, then $A_r=B_{r-1}+\langle e_r,a_{\lambda,0}\mid \mu(a_{\lambda,0})=r\rangle$;
\item $B_r=A_r+\langle a_{\lambda,k}\mid \mu(a_{\lambda,k})=r\rangle$.
\end{itemize}

By construction, we have $A_0=B_0\subseteq A_1\subseteq B_1\subseteq A_2\subseteq\cdots$; in particular, $\{B_r\}_{r=0}^\infty$ is an ascending chain of subgroups of $G$. The image of $\{a_{\lambda,n}\mid\lambda\in\Lambda,n\inN\}$ in $\kappa_\infty(G)$ generates $\kappa_\infty(G)$ (since $\{b_\lambda\}_{\lambda\in\Lambda}$ is a $\insQ$-basis of $\kappa_\infty(G)$); since every $e_k$ and every $a_{\lambda,k}$ are in $B_r$ for some $r$, by Corollary \ref{cor:generatori-ex-A} we have $\bigcup_r B_r=G$. We want to show that every $B_r$ and every quotient $B_r/B_{r-1}$ are free groups, and we proceed by induction.

We first show that $A_0=B_0$ is free, and we do so by showing that $\{e_0,a_{\lambda,0}\mid \mu(a_{\lambda,0})=0\}$ is a basis. Suppose not: then, there is a relation of linear dependence
\begin{equation*}
ne_0+\sum_{i=1}^kn_ia_{\lambda_i,0}=0.
\end{equation*}
Applying $\pi_\infty$ to this relation, we obtain $\sum_{i=1}^kn_i\pi_\infty(a_{\lambda_i,0})=0$; however, by construction, the set $\{\pi_\infty(a_{\lambda_i,0})\}$ is linearly independent over $\insQ$, and thus $n_i=0$ for every $i$. This fact leaves $ne_0=0$, which implies $n=0$. Hence $A_0$ is free.

To prove that $B_r/B_{r-1}$ is free, we consider the exact sequence
\begin{equation}\label{eq:Br-succext}
0\longrightarrow\frac{A_r}{B_{r-1}}\longrightarrow\frac{B_r}{B_{r-1}}\longrightarrow\frac{B_r}{A_r}\longrightarrow 0.
\end{equation}
We first show that $A_r/B_{r-1}$ is free, and we proceed as above by showing that the image of $\{e_r,a_{\lambda,0}\mid \mu(a_{\lambda,0})=r\}$ is a basis. If not, there is a relation of linear dependence
\begin{equation*}
ne_r+\sum_{i=1}^kn_ia_{\lambda_i,0}\in B_{r-1}.
\end{equation*}
Applying $\pi_\infty$, we obtain $\sum_{i=1}^kn_i\pi_\infty(a_{\lambda_i,0})\in\pi_\infty(B_{r-1})$; however, $\pi_\infty(B_{r-1})$ is contained in the $\insQ$-linear space generated by a set of $a_{\lambda,0}$ with $\lambda\neq\lambda_i$. As $\{\pi_\infty(a_{\lambda,0})\mid \lambda\in\Lambda\}$ is $\insQ$-linearly independent, it follows that $n_i=0$ for all $i$ and thus also $ne_r\in B_{r-1}$. Therefore, we can write
\begin{equation*}
ne_r=\sum_{i=0}^{r-1}m_ie_i+\sum_{i,j}n_{i,j}a_{\lambda_i,j}
\end{equation*}
for some $m_i,n_{i,j}\inZ$ and some $a_{\lambda_i,j}$ with $\mu(a_{\lambda_i,j})<r$. By construction, $r!a_{\lambda_i,j}=s_{i,j}a_{\lambda_i,0}+h_{i,j}$, with $s_{i,j}\inZ$ and $h_{i,j}(k)=0$ for all $k\geq\mu(a_{\lambda_i,j})$ (in particular, $h_{i,j}\in\ker\pi_\infty$). Therefore,
\begin{equation}\label{eq:ne_r}
r!ne_r=\sum_{i=0}^{r-1}m_ie_i+\sum_{i,j}h_{i,j}+\sum_{i,j}s_{i,j}a_{\lambda_i,0}
\end{equation}
Applying again $\pi_\infty$, we see that the rightmost sum of the previous equation is $0$; by the linear independence of the $\pi_\infty(a_{\lambda_i,0})=b_{\lambda_i}$, we obtain that
\begin{equation*}
\pi_\infty\left(\sum_js_{i,j}a_{\lambda_i,0}\right)=0
\end{equation*}
for every $i$, and thus
\begin{equation*}
\sum_js_{i,j}=0
\end{equation*}
for all $i$. It follows that the rightmost sum of \eqref{eq:ne_r} is $0$, and so 
\begin{equation*}
r!ne_r=\sum_{i=0}^{r-1}m_ie_i+\sum_{i,j}h_{i,j}.
\end{equation*}
By construction, $h_{i,j}(k)=0$ for all $k\geq\mu(a_{\lambda_i,j})$, and the latter is at most $r-1$; in particular, $h_{i,j}(r)=0$ for all $r$. Evaluating the previous equality at $r$, we obtain $r!n=0+0$. Hence $n=0$. It follows that $A_r/B_{r-1}$ is free.

We now show that that $B_r/A_r$ is a group of bounded torsion. Indeed, its generators are the images of $a_{\lambda,k}$ for some $\lambda$ and some $k$; the latter, by construction, is at least $1$ and at most $r$. Therefore,
\begin{equation*}
d_{\lambda,k}a_{\lambda,k}-a_{\lambda,0}\in\langle e_i\mid i=0,\ldots r\rangle\subseteq A_r
\end{equation*}
and $d_{\lambda,k}a_{\lambda,k}\in A_r$. Since $d_{\lambda,k}|k!$ and $k\leq r$, we also have $r!B_r\subseteq A_r$, or equivalently $r!(B_r/A_r)=0$.

Using these properties and Lemma \ref{lemma:free-boundtors}, the exact sequence \eqref{eq:Br-succext} implies that $B_r/B_{r-1}$ is a free group. Since $B_0$ is free, the projectivity of the free groups implies that each $B_r$ is free. Therefore, $\{B_r\}_{r=0}^\infty$ is an ascending sequence of free subgroups of $G$ with free quotients, and thus by Theorem \ref{teor:free-chain} their union is free. As their union is $G$, we have that $G$ is free.
\end{proof}

\section{The limit ordinal case}\label{sect:limit}
In this section we consider basic limit groups where the Cantor-Bendixson rank of the prime at infinity is a limit ordinal. We distinguish two subcases, depending on the cofinality of the rank.

\begin{teor}\label{teor:limit-countcof}
Let $G\subseteq\funct(X,\insZ)$ be a one-dimensional basic limit group, and let $P$ be its prime at infinity. Suppose that:
\begin{itemize}
\item $P\notin\deriv^\infty(\Specast(G))$;
\item $\cb(P)$ is a limit ordinal of countable cofinality;
\item $\kappa_\infty(G)$ has a $\insQ$-basis.
\end{itemize}
Then, $G$ is free.
\end{teor}
\begin{proof}
We can suppose without loss of generality that $G$ is reduced (Proposition \ref{prop:reduction-reduced}) and that $P$ is the unique prime of Cantor-Bendixson rank equal to $\cb(\Specast(G))$ (Proposition \ref{prop:reduction-maxcb}). In particular, $\Specast(G)$ is scattered.

For every $x\in X$, let $q_x$ be a semibasic element at $x$; such an element exists because $\Specast(G)$ is scattered (Lemma \ref{lemma:quark}\ref{lemma:quark:exists}).

Choose a countable sequence $\{\alpha_n\}_{n\inN}$ of ordinal numbers with $\alpha_0>0$ whose limit is $\cb(P)$.

Let $\residuo_\infty(G)\subseteq\insR$ be the residue group at infinity of $G$ and let $\{b_\lambda\}_{\lambda\in\Lambda}$ be a $\insQ$-basis of $\residuo_\infty(G)$.

Fix any $\lambda$ and let $H_\lambda:=b_\lambda\insQ\cap\residuo_\infty(G)$. Since $H_\lambda$ is isomorphic to a subgroup of $\insQ$, we can find a sequence $\{\widetilde{f}_{\lambda,n}\}_{n\inN}$ such that, for every $n\inN$,
\begin{itemize}
\item $\pi_\infty(\widetilde{f}_{\lambda,0})=b_\lambda$;
\item $\{\pi_\infty(\widetilde{f}_{\lambda,n})\}_{n\inN}$ generates $H_\lambda$;
\item $n!\pi_\infty(\widetilde{f}_{\lambda,n})\in b_\lambda\insZ$.
\end{itemize}
We now construct a new sequence $\{f_{\lambda,n}\}$ from $\{\widetilde{f}_{\lambda,n}\}_{n\inN}$.

When $n=0$, we set $f_{\lambda,0}:=\widetilde{f}_{\lambda,0}$.

If $n>0$, then, by the third bullet point, there is a (unique) $t\inZ$ such that $n!\pi_\infty(\widetilde{f}_{\lambda,n})=t\pi_\infty(f_{\lambda,0})$; let $g_{\lambda,n}:=n!\widetilde{f}_{\lambda,n}-tf_{\lambda,0}$. By construction, $g_{\lambda,n}\in\ker\pi_\infty$; we distinguish two cases.

If $g_{\lambda,n}=0$ or $\cb(g_{\lambda,n})<\alpha_n$, then we take any $x\in X$ with $\cb(x)>\alpha_n$, and set $f_{\lambda,n}:=\widetilde{f}_{\lambda,n}+q_x$. 

If $\cb(g_{\lambda,n})>\alpha_n$, we set $f_{\lambda,n}:=\widetilde{f}_{\lambda,n}$.

In both cases, $\pi_\infty(f_{\lambda,n})=\pi_\infty(\widetilde{f}_{\lambda,n})$, so the sequence $\{f_{\lambda,n}\}_{n\inN}$ satisfies the same properties as above, and in addition we have $\alpha_n<\cb(n!f_{\lambda,n}-tf_{\lambda,0})<\cb(\Specast(G))$.

We now define two well-ordered sequences $\{A_\beta\}_\beta$ and $\{B_\beta\}_\beta$ of subgroups of $G$ recursively.
\begin{itemize}
\item $A_0=\langle q_x,f_{\lambda_0}\mid \cb(x)=0\rangle$;
\item $B_0=A_0$;
\item if $\beta$ is a limit ordinal,
\begin{equation*}
A_\beta=B_\beta:=\bigcup_{\alpha<\beta}B_\alpha;
\end{equation*}
\item if $\beta=\gamma+1$ is a successor ordinal,
\begin{equation*}
\begin{aligned}
A_\beta & =B_\gamma+\langle q_x\mid \cb(x)=\beta\rangle,\\
B_\beta & =A_\beta+\langle f_{\lambda,n}\mid \cb(n!f_{\lambda,n}-tf_{\lambda,0})=\beta\rangle.
\end{aligned}
\end{equation*}
\end{itemize}
By construction, $A_0=B_0\subseteq A_1\subseteq B_1\subseteq A_2\subseteq\cdots\subseteq A_\beta\subseteq B_\beta\subseteq\cdots$;  in particular, $\{B_\beta\}_\beta$ is an ascending chain of subgroups of $G$. As in the proof of Theorem \ref{teor:onelimitpoint}, $\bigcup_\beta B_\beta=G$ since the union contains all $q_x$ and a family that generates the whole $\residuo_\infty(G)$ (since $\{b_\lambda\}$ is a $\insQ$-basis; Corollary \ref{cor:generatori-ex-A}). We want to show that every $B_\beta$ and every quotient $B_{\beta+1}/B_\beta$ is a free group, and we proceed by induction.

To show that $A_0=B_0$ is free we show that $\{q_x,f_{\lambda_0}\mid \cb(x)=0\}$ is linearly independent. If not, there are $x_i$ with $\cb(x_i)=0$, $\lambda_j$, $n_i,m_j\in\insZ$ such that
\begin{equation*}
\sum_{i=1}^rn_iq_{x_i}+\sum_{j=1}^sm_jf_{\lambda_j,0}=0.
\end{equation*}
Applying $\pi_\infty$, we obtain
\begin{equation*}
\sum_{j=1}^sm_jb_{\lambda_j}=0
\end{equation*}
which implies $m_j=0$ for all $j$ since the $b_{\lambda_j}$ are linearly independent. Thus $\sum_{i=1}^rn_iq_{x_i}=0$; since $\supp(q_x)=\{x\}$ whenever $\cb(x)=0$ we also have $n_i=0$ for all $i$. Hence $A_0=B_0$ is free.

To show that $B_{\beta+1}/B_\beta$ is free, we consider the exact sequence
\begin{equation}\label{eq:Br-succext-limit1}
0\longrightarrow\frac{A_{\beta+1}}{B_\beta}\longrightarrow\frac{B_{\beta+1}}{B_\beta}\longrightarrow\frac{B_{\beta+1}}{A_{\beta+1}}\longrightarrow 0.
\end{equation}
We first consider $B_{\beta+1}/A_{\beta+1}$. Let $n$ be the smallest integer such that $\beta+1\leq\alpha_n$; such an $n$ exists by the choice of the sequence $\{\alpha_n\}$. Take $f_{\lambda,i}\in A_{\beta+1}$: by construction $i\leq n$ and $\cb(i!f_{\lambda,i}-tf_{\lambda,0})=\beta$ for some $t\inZ$; therefore, 
\begin{equation*}
n!f_{\lambda,i}-t\frac{n!}{i!}f_{\lambda,0}
\end{equation*}
is an element of $G$ that belongs to $A_{\beta+1}$. Hence $n!f_{\lambda,i}\in A_{\beta+1}$. Since $n$ is independent of $\lambda$, we have $n!(B_{\beta+1}/A_{\beta+1})=0$, i.e., $B_{\beta+1}/A_{\beta+1}$ is a group of bounded torsion.

We now show that $A_{\beta+1}/B_\beta$ is free, and we proceed as above by showing that the image of $\{q_x\mid \cb(x)=\beta+1\}$ is a basis. If not, there are $a_1,\ldots,a_k\inZ$ and $x_1,\ldots,x_k\in X$ with $\cb(x_i)=\beta+1$ such that $a_1q_{x_1}+\cdots+a_kq_{x_k}\in B_\beta$. Therefore, we can write
\begin{equation}\label{eq:abeta-quoz-limit-cofcount}
a_1q_{x_1}+\cdots+a_kq_{x_k}=\sum_{i=1}^tc_iq_{y_i}+\sum_{j=1}^sd_jf_{\lambda_j,n_j}
\end{equation}
for some $y_i$ with $\cb(y_i)\leq\beta$ and some $c_i,d_j\inZ$ (note that these coefficients may not be unique) and some $f_{\lambda_j,n_j}\in B_\beta$. Let $n$ be the maximum of the $n_i$ and multiply the previous equality by $n!$. Then, each summand $n!f_{\lambda_j,n_j}$ is equal to $t_jf_{\lambda_j,0}+h_j$, where $\pi_\infty(h_j)=0$ and $\cb(h)\leq\beta$. By Lemma \ref{lemma:spanqx}, $h_j$ can be written as a sum of elements $c_iq_{z_i}$ with $q_{z_i}\in A_\beta\subseteq B_\beta$: thus we can suppose without loss of generality that in \eqref{eq:abeta-quoz-limit-cofcount} all $n_j=0$. Applying $\pi_\infty$ to both sides, we have an equality
\begin{equation*}
0=0+\sum_{j=1}^sd_jb_{\lambda_j};
\end{equation*}
the linear independence of the $b_\lambda$ now implies that $d_j=0$ for all $j$. Thus
\begin{equation*}
a_1q_{x_1}+\cdots+a_kq_{x_k}=\sum_{i=1}^tc_iq_{y_i}.
\end{equation*}
Since $\{q_{x_1},\ldots,q_{x_k},q_{y_1},\ldots,q_{y_t}\}$ is a family of semibasic elements at different points of $X$, by Lemma \ref{lemma:quark-linindip}, we must have $a_1=\cdots=a_k=0$, and thus $A_{\beta+1}/B_\beta$ is free.

Using these properties and Lemma \ref{lemma:free-boundtors}, the exact sequence \eqref{eq:Br-succext-limit1} implies that $B_{\beta+1}/B_\beta$ is a free group. Since $B_0$ is free, the projectivity of the free groups implies that each $B_{\beta+1}$ is free. Therefore, $\{B_\beta\}_{r=0}^\infty$ is a smooth ascending sequence of free subgroups of $G$ with free quotients, and thus by Theorem \ref{teor:free-chain} their union is free. By Corollary \ref{cor:generatori-ex-A}, their union is $G$; thus $G$ is free.
\end{proof}

\begin{teor}\label{teor:limit-uncountcof}
Let $G\subseteq\funct(X,\insZ)$ be a one-dimensional basic limit group, and let $P$ be its prime at infinity. Suppose that:
\begin{itemize}
\item $P\notin\deriv^\infty(\Specast(G))$;
\item $\cb(P)$ is a limit ordinal of uncountable cofinality;
\item $\kappa_\infty(G)$ has a $\insQ$-basis.
\end{itemize}
Then, $G$ is free.
\end{teor}
\begin{proof}
The proof is very similar to the one of Theorem \ref{teor:limit-countcof}.

We can suppose without loss of generality that $G$ is reduced (Proposition \ref{prop:reduction-reduced}) and that $P$ is the unique prime of Cantor-Bendixson rank equal to $\cb(\Specast(G))$ (Proposition \ref{prop:reduction-maxcb}), so in particular $\Specast(G)$ is scattered.

For every $x\in X$, let $q_x$ be a semibasic element at $x$; these elements exist by Lemma \ref{lemma:quark}\ref{lemma:quark:exists}.

Let $\residuo_\infty(G)\subseteq\insR$ be the residue group at infinity of $G$. Let $\{b_\lambda\}_{\lambda\in\Lambda}$ be a $\insQ$-basis of $\residuo_\infty(G)$.

Fix any $\lambda$ and let $H_\lambda:=b_\lambda\insQ\cap\residuo_\infty(G)$. Since $H_\lambda$ is isomorphic to a subgroup of $\insQ$, we can find a sequence $\{\widetilde{f}_{\lambda,n}\}_{n\inN}$ such that, for every $n\inN$,
\begin{itemize}
\item $\pi_\infty(\widetilde{f}_{\lambda,0})=b_\lambda$;
\item $\{\pi_\infty(\widetilde{f}_{\lambda,n})\}_{n\inN}$ generates $H_\lambda$;
\item $n!\pi_\infty(\widetilde{f}_{\lambda,n})\in b_\lambda\insZ$.
\end{itemize}
We now construct a new sequence $\{f_{n,\lambda}\}$ from $\{\widetilde{f}_{\lambda,n}\}_{n\inN}$.

When $n=0$, we set $f_{\lambda,0}:=\widetilde{f}_{\lambda,0}$.

If $n>0$, then, by the third bullet point, there is a (unique) $t\inZ$ such that $n!\pi_\infty(\widetilde{f}_{\lambda,n})=t\pi_\infty(f_{\lambda,0})$; let $g_{\lambda,n}:=n!\widetilde{f}_{\lambda,n}-tf_{\lambda,n}$. By construction, $g_{\lambda,n}\in \ker\pi_\infty$. If $g_{\lambda,n}\neq 0$, define 
\begin{equation*}
\theta_{\lambda,n}:=\cb(g_{\lambda,n});
\end{equation*}
then, $\theta_{\lambda,n}<\cb(\Specast(G))$ since $\cb(g_{\lambda,n})$ is a successor ordinal.

Since $\cb(P)$ has uncountable cofinality while $\{\theta_{\lambda,n}\}_{n=1}^\infty$ is countable, there is a limit ordinal $\theta'_\lambda$ such that $\theta_{\lambda,n}<\theta'_\lambda<\cb(P)$ for all $n$ for which $\theta_{\lambda,n}$ is defined; choose an $x\in X$ with $\cb(x)=\theta'_\lambda+n$ and set $f_{\lambda,n}:=\widetilde{f}_{\lambda,n}+q_x$. Then, $\pi_\infty(f_{\lambda,n})=\pi_\infty(\widetilde{f}_{\lambda,n})$, so the sequence $\{f_{\lambda,n}\}_{n\inN}$ satisfies the same properties as above, and in addition we have that $\cb(n!f_{\lambda,n}-tf_{\lambda,0})$ is always equal to a limit ordinal plus $n$.

We now define two well-ordered sequences $\{A_\beta\}_\beta$ and $\{B_\beta\}_\beta$ of subgroups of $G$ recursively.
\begin{itemize}
\item $A_0=\langle q_x,f_{\lambda_0}\mid \cb(x)=0\rangle$;
\item $B_0=A_0$;
\item if $\beta$ is a limit ordinal,
\begin{equation*}
A_\beta=B_\beta:=\bigcup_{\alpha<\beta}B_\alpha;
\end{equation*}
\item if $\beta=\gamma+1$ is a successor ordinal,
\begin{equation*}
\begin{aligned}
A_\beta & =B_\gamma+\langle q_x\mid \cb(x)=\beta\rangle,\\
B_\beta & =A_\beta+\langle f_{\lambda,i}\mid \cb(n!f_{\lambda,n}-tf_{\lambda,0})=\beta\rangle.
\end{aligned}
\end{equation*}
\end{itemize}
By construction, $A_0=B_0\subseteq A_1\subseteq B_1\subseteq A_2\subseteq\cdots\subseteq A_\beta\subseteq B_\beta\subseteq\cdots$;  in particular, $\{B_\beta\}_\beta$ is an ascending chain of subgroups of $G$.  As in the proof of Theorem \ref{teor:onelimitpoint}, $\bigcup_\beta B_\beta=G$ since the union contains all $q_x$ and a family that generates the whole $\residuo_\infty(G)$ (since $\{b_\lambda\}$ is a $\insQ$-basis; Corollary \ref{cor:generatori-ex-A}). We want to show that every $B_\beta$ and every quotient $B_{\beta+1}/B_\beta$ is a free group, and we proceed by induction.

The fact that $A_0=B_0$ is free follows as in the proof of Theorem \ref{teor:limit-countcof}.

To show that $B_{\beta+1}/B_\beta$ is free, we consider the exact sequence
\begin{equation}\label{eq:Br-succext-limit2}
0\longrightarrow\frac{A_{\beta+1}}{B_\beta}\longrightarrow\frac{B_{\beta+1}}{B_\beta}\longrightarrow\frac{B_{\beta+1}}{A_{\beta+1}}\longrightarrow 0.
\end{equation}
We first consider $B_{\beta+1}/A_{\beta+1}$. If $f_{\lambda,n}$ gets added in the passage from $A_{\beta+1}$ to $B_{\beta+1}$, it follows that $\beta+1=\theta+n$ for some limit ordinal $\theta$; in particular, all such new elements share the same $n$. It follows that $n!f_{\lambda,n}-tf_{\lambda,0}\in A_{\beta+1}$ for some $t$ (depending on $\lambda$ and $n$), and thus $n!f_{\lambda,n}\in A_{\beta+1}$. In particular, $n!B_{\beta+1}\subseteq A_{\beta+1}$, i.e., $B_{\beta+1}/A_{\beta+1}$ is a group of bounded torsion.

The proof that $A_{\beta+1}/B_\beta$ is free is exactly equal to corresponding part of the proof of Theorem \ref{teor:limit-countcof}.

Using these properties and Lemma \ref{lemma:free-boundtors}, the exact sequence \eqref{eq:Br-succext-limit2} implies that $B_{\beta+1}/B_\beta$ is a free group. Since $B_0$ is free, the projectivity of the free groups implies that each $B_{\beta+1}$ is free. Therefore, $\{B_\beta\}_{r=0}^\infty$ is a smooth ascending sequence of free subgroups of $G$ with free quotients, and thus by Theorem \ref{teor:free-chain} their union is free. By Corollary \ref{cor:generatori-ex-A}, their union is $G$; thus $G$ is free.
\end{proof}

The previous theorems can be summarized in the following:
\begin{teor}\label{teor:basic}
Let $G\subseteq\funct(X,\insZ)$ be a one-dimensional basic $\ell$-group with prime at infinity $P$. Suppose that:
\begin{itemize}
\item $\Specast(G)$ is Fr\'echet-Urysohn;
\item $P\notin\deriv^\infty(\Specast(G))$; and
\item $\kappa_\infty(G)$ has a $\insQ$-basis,
\end{itemize}
Then, $G$ is free.
\end{teor}

\section{Back to ring theory}
Theorem \ref{teor:basic} can be immediately transferred to the domain case as the following result:
\begin{teor}\label{teor:invfree-1ptolim}
Let $D$ be a one-dimensional dd-domain. Suppose that:
\begin{itemize}
\item $\Max(D)$ is Fr\'echet-Urysohn;
\item $D_P$ is a DVR for all $P\in\Max(D)\setminus\{Q\}$;
\item $Q\notin\deriv^\infty(\Max(D))$;
\item the value group of $D_Q$ has a $\insQ$-basis.
\end{itemize}
Then, $\Inv(D)$ is free.
\end{teor}
\begin{proof}
The set $X:=\Max(D)\setminus\{Q\}$ is a discrete core, and $\Inv(D)$ is a subgroup $G$ of $\funct(X,\insZ)$. The hypothesis now imply that $G$ is a one-dimensional basic limit group, and that it satisfies the hypothesis of Theorem \ref{teor:basic}. Hence $\Inv(D)$ is free.
\end{proof}

We now want to improve the previous theorem to the case where more than one localization is not a DVR. We shall use the following notation: if $D$ is a Pr\"ufer domain $A\subseteq\Max(D)$, then
\begin{equation*}
\Theta(A):=\bigcap_{P\in A}D_P.
\end{equation*}
The ring $\Theta(A)$ is always a one-dimensional Pr\"ufer domain; its maximal ideals are the extensions of the maximal ideals of $D$ contained in the closure of $A$, with respect to the inverse topology.

We shall also need the following result.
\begin{prop}\label{prop:unione-numerabile}
\cite[Theorem 4.11]{overring-operators} Let $D$ be a one-dimensional Pr\"ufer domain, and let $\{A_n\}_{n\inN}$ be a countable set of clopen subsets of $\Max(D)$; let $A:=\bigcup_nA_n$. If $\Inv(\Theta(A_n))$ is free for every $n$, then the kernel of the extension map
\begin{equation*}
\Inv(D)\longrightarrow\Inv(\Theta(\Max(D)\setminus A))
\end{equation*}
is free.
\end{prop}

\begin{teor}\label{teor:invfree-countptilim}
Let $D$ be a one-dimensional dd-domain, and let $\mathcal{N}:=\{P\in\Max(D)\mid P$ is not a DVR$\}$. Suppose that:
\begin{itemize}
\item $\mmax:=\Max(D)$ is Fr\'echet-Urysohn;
\item $\mathcal{N}$ is countable;
\item $\mathcal{N}\cap\deriv^\infty(\mmax)=\emptyset$;
\item for every $P\in \mathcal{N}$, the value group of $D_P$ has a $\insQ$-basis.
\end{itemize}
Then, $\Inv(D)$ is free.
\end{teor}
\begin{proof}
Let $\mathcal{N}:=\{M_0,M_1,\ldots,M_n,\ldots\}$. Note that $\mathcal{N}$ is scattered since it does not meet $\deriv^\infty(\mmax)$. We proceed by induction on $\alpha=\max\{\beta\mid \mathcal{N}\cap\deriv^\beta(\mmax)\neq\emptyset\}$. We note that $\alpha\neq 0$ since $\mmax\setminus\mathcal{N}$ is dense and thus no point of $\mathcal{N}$ is isolated in $\mmax$.

If $\alpha=1$, then all elements of $\mathcal{N}$ have Cantor-Bendixson rank $1$; thus $\mathcal{N}$ is discrete. Hence for every $n$ we can find a clopen subset $A_n$ such that $A_n\cap\mathcal{N}=\{M_n\}$. By Theorem \ref{teor:invfree-1ptolim}, every $\Inv(\Theta(A_n))$ is free; by Proposition \ref{prop:unione-numerabile}, it follows that the kernel of $\Inv(D)\longrightarrow\Inv(\Theta(\mmax\setminus A))$ is free, where $A:=\bigcup_nA_n$. The ring $\Inv(\Theta(\mmax\setminus A))$ is an almost Dedekind domain, since $D_P$ is a DVR for every $P\in\mmax\setminus X\subseteq\mmax\setminus A$; hence $\Inv(\Theta(\mmax\setminus A))$ is free. It follows that $\Inv(D)$ is free.

Suppose that the claim is true for all ordinal numbers $\beta<\alpha$. If $\alpha$ is a limit ordinal, then the Cantor-Bendixson rank of every $M\in\mathcal{N}$ is strictly smaller than $\alpha$; it follows that for every $M_n$ there is a clopen subset $A_n$ of $\mmax$ such that $M_n$ is the unique element of $A_n\cap\mathcal{N}$ whose Cantor-Bendixson rank is equal to $\cb(M_n)$. Consider $\Theta(A_n)$: then, its maximal ideals $P$ such that $(A_n)_P$  is not a DVR are extensions of the maximal ideals in $\mathcal{N}\cap A_n$. In particular, the Cantor-Bendixson rank of these elements is $\leq\cb(M_n)<\alpha$, and thus $\Inv(\Theta(A_n))$ is free by inductive hypothesis. The claim now follows as in the case $\alpha=0$.

Suppose now that $\alpha$ is a successor ordinal, and let $Y_\alpha:=\{P\in\mathcal{N}\mid \cb(P)=\alpha\}$. Suppose first that $Y_\alpha$ is a singleton; without loss of generality, say $Y_\alpha=\{M_0\}$. Since $\mmax$ is Fr\'echet-Urysohn we can find a nonempty subset $C\subseteq\mmax\setminus\mathcal{N}$ such that $\overline{C}\cap\mathcal{N}=\{M_0\}$. For every $n\geq 1$, let $A_n$ be a clopen subset of $\mmax$ such that $A_n\cap C=\emptyset$ and $M_n$ is the unique element of $A_n\cap\mathcal{N}$ whose Cantor-Bendixson rank is $\cb(M_n)$. Then, as above,  $\mathcal{N}_n:=\{P\in\Max(\Theta(A_n))\mid D_P$ is not a DVR$\}$ contains only elements whose Cantor-Bendixson rank (in $\Max(\Theta(A_n))$ is at most $\cb(M_n)<\alpha$; by inductive hypothesis, $\Inv(\Theta(A_n))$ is free. Let $A:=\bigcup_nA_n$: by Proposition \ref{prop:unione-numerabile}, the kernel of the extension map   $\Inv(D)\longrightarrow\Inv(\Theta(\mmax\setminus A))$ is free. In $\Theta(\mmax\setminus A)$, there is a unique prime ideal $P$ such that $\Omega(A)_P$ is not a DVR, namely $M_0\Theta(\mmax\setminus A)$; furthermore, this prime is not an isolated point of $\mmax$ since $C\cap A=\emptyset$ and $M_0$ is in the closure of $C$. By Theorem \ref{teor:invfree-1ptolim}, $\Inv(\Theta(\mmax\setminus A))$ is free; thus $\Inv(D)$ is free too.

Suppose that $Y_\alpha$ is arbitrary. The set $Y_\alpha$ is discrete and thus finite, say $Y_\alpha=\{N_1,\ldots,N_k\}$; thus we can find clopen subsets $A_n$ of $\mmax$ such that $A_n\cap Y_\alpha=\{N_n\}$ for every $1\leq n\leq k$. The proof now follows as in the case $\alpha=0$, using the case in which $Y_\alpha$ is a singleton. 
\end{proof}

We transfer back the previous theorem to the realm of $\ell$-groups.
\begin{teor}\label{teor:grp-countptilim}
Let $G\subseteq\funct(X,\insZ)$ be a one-dimensional $\ell$-group, and let $\mathcal{N}:=\{P\in\Specast(G)\mid G/P$ is not isomorphic to $\insZ\}$. Suppose that:
\begin{itemize}
\item $\Specast(G)$ is Fr\'echet-Urysohn;
\item $\mathcal{N}$ is countable;
\item $\mathcal{N}\cap\deriv^\infty(\Specast(G))=\emptyset$;
\item for every $P\in \mathcal{N}$, the residue group $G/P$ has a $\insQ$-basis.
\end{itemize}
Then, $G$ is free.
\end{teor}

\bibliographystyle{plain}
\bibliography{/bib/articoli,/bib/libri,/bib/miei}

\begin{thebibliography}{10}

\bibitem{anderson-feil}
Marlow Anderson and Todd Feil.
\newblock {\em Lattice-ordered groups}.
\newblock Reidel Texts in the Mathematical Sciences. D. Reidel Publishing Co.,
  Dordrecht, 1988.
\newblock An introduction.

\bibitem{baer-dirprodZ}
Reinhold Baer.
\newblock Abelian groups without elements of finite order.
\newblock {\em Duke Math. J.}, 3(1):68--122, 1937.

\bibitem{brewer-klinger}
James Brewer and Lee Klingler.
\newblock The ordered group of invertible ideals of a {P}r\"ufer domain of
  finite character.
\newblock {\em Comm. Algebra}, 33(11):4197--4203, 2005.

\bibitem{chabert-fatou}
Jean-Luc Chabert.
\newblock Anneaux de ``polyn\^omes \`a{} valeurs enti\`eres'' et anneaux de
  {F}atou.
\newblock {\em Bull. Soc. Math. France}, 99:273--283, 1971.

\bibitem{darnel-lgroups}
Michael~R. Darnel.
\newblock {\em Theory of lattice-ordered groups}, volume 187 of {\em Monographs
  and Textbooks in Pure and Applied Mathematics}.
\newblock Marcel Dekker, Inc., New York, 1995.

\bibitem{spectralspaces-libro}
Max Dickmann, Niels Schwartz, and Marcus Tressl.
\newblock {\em Spectral spaces}, volume~35 of {\em New Mathematical
  Monographs}.
\newblock Cambridge University Press, Cambridge, 2019.

\bibitem{dobbs_fedder_fontana}
David~E. Dobbs, Richard Fedder, and Marco Fontana.
\newblock Abstract {R}iemann surfaces of integral domains and spectral spaces.
\newblock {\em Ann. Mat. Pura Appl. (4)}, 148:101--115, 1987.

\bibitem{fuchs-abeliangroups}
L\'{a}szl\'{o} Fuchs.
\newblock {\em Abelian groups}.
\newblock Springer Monographs in Mathematics. Springer, Cham, 2015.

\bibitem{fuchs-salce}
L\'aszl\'o{} Fuchs and Luigi Salce.
\newblock {\em Modules over non-{N}oetherian domains}, volume~84 of {\em
  Mathematical Surveys and Monographs}.
\newblock American Mathematical Society, Providence, RI, 2001.

\bibitem{gilmer}
Robert Gilmer.
\newblock {\em Multiplicative {I}deal {T}heory}.
\newblock Marcel Dekker Inc., New York, 1972.
\newblock Pure and Applied Mathematics, No. 12.

\bibitem{henriksen_prime}
Melvin Henriksen.
\newblock On the prime ideals of the ring of entire functions.
\newblock {\em Pacific J. Math.}, 3:711--720, 1953.

\bibitem{HK-Olb-Re}
Olivier~A. Heubo-Kwegna, Bruce Olberding, and Andreas Reinhart.
\newblock Group-theoretic and topological invariants of completely integrally
  closed {P}r\"{u}fer domains.
\newblock {\em J. Pure Appl. Algebra}, 220(12):3927--3947, 2016.

\bibitem{olberding-factoring-SP}
Bruce Olberding.
\newblock Factorization into radical ideals.
\newblock In {\em Arithmetical properties of commutative rings and monoids},
  volume 241 of {\em Lect. Notes Pure Appl. Math.}, pages 363--377. Chapman \&
  Hall/CRC, Boca Raton, FL, 2005.

\bibitem{olberding_affineschemes}
Bruce Olberding.
\newblock Affine schemes and topological closures in the {Z}ariski-{R}iemann
  space of valuation rings.
\newblock {\em J. Pure Appl. Algebra}, 219(5):1720--1741, 2015.

\bibitem{specker}
Ernst Specker.
\newblock Additive {G}ruppen von {F}olgen ganzer {Z}ahlen.
\newblock {\em Portugal. Math.}, 9:131--140, 1950.

\bibitem{overring-operators}
Dario Spirito.
\newblock Smooth chains and overring operators.
\newblock {\em submitted}, arXiv:2511.15319.

\bibitem{SP-scattered}
Dario Spirito.
\newblock Almost {D}edekind domains without radical factorization.
\newblock {\em Forum Math.}, 35(2):363--382, 2023.

\bibitem{InvXD}
Dario Spirito.
\newblock Radical factorization in higher dimension.
\newblock {\em J. Pure Appl. Algebra}, 229(11):108111, 2025.

\bibitem{bounded-almded}
Dario Spirito.
\newblock Boundness in almost {D}edekind domains.
\newblock {\em J. Algebra Appl.}, to appear.

\end{thebibliography}
\end{document}